\documentclass{amsart}
\usepackage{amssymb,latexsym}
\usepackage{amsmath}
\usepackage{graphicx}
\usepackage{color}

\newtheorem{thm}{Theorem}[section]
\newtheorem{conj}[thm]{Conjecture}
\newtheorem{cor}[thm]{Corollary}
\newtheorem{defin}[thm]{Definition}
\newtheorem{lemma}[thm]{Lemma}
\newtheorem{prop}[thm]{Proposition}

\newcommand{\bdd}{\mbox{$\partial$}}

\newcommand{\aaa}{\mbox{$\alpha$}}
\newcommand{\bbb}{\mbox{$\beta$}}

\newcommand{\sss}{\mbox{$\sigma$}}

\newcommand{\cW}{\mbox{$\mathcal{W}$}}
\newcommand{\cV}{\mbox{$\mathcal{V}$}}
\newcommand{\oV}{\mbox{$\overline{V}$}}
\newcommand{\oW}{\mbox{$\overline{W}$}}
\newcommand{\cWp}{\mbox{$\mathcal{W_+}$}}
\newcommand{\cVp}{\mbox{$\mathcal{V_+}$}}
\newcommand{\cWm}{\mbox{$\mathcal{W_-}$}}
\newcommand{\cVm}{\mbox{$\mathcal{V_-}$}}
\newcommand{\cv}{\mbox{$\nu$}}
\newcommand{\cw}{\mbox{$\omega$}}
\newcommand{\cF}{\mbox{$\mathcal{F}$}}
\newcommand{\cT}{\mbox{$\mathcal{T}$}}
\newcommand{\No}{\mbox{$\mathbb{N}$}}
\newcommand{\Np}{\mbox{$\mathbb{N} \cup \{0 \}$}}
\newcommand{\hP}{\mbox{$\hat{P}$}}
\newcommand{\hF}{\mbox{$\hat{F}$}}
\newcommand{\hv}{\mbox{$\hat{v}$}}
\newcommand{\hw}{\mbox{$\hat{w}$}}
\newcommand{\hcv}{\mbox{$\hat{\cv}$}}
\newcommand{\hcw}{\mbox{$\hat{\cw}$}}

\begin{document}


\keywords {Heegaard splitting, Heegaard stabilization}

\thanks{Second author partially supported by an NSF grant.}

\title{A proof of the Gordon Conjecture}

\author{Ruifeng Qiu}
\address{\hskip-\parindent
Ruifeng Qiu\\
Department of Mathematics\\
Dalian University of Technology \\
Dalian 116024, China}
\email{Email: qiurf@dlut.edu.cn}

\author{Martin Scharlemann}
\address{\hskip-\parindent
Martin Scharlemann\\
Mathematics Department \\
University of California, Santa Barbara \\
Santa Barbara, CA 93117, USA}
\email{mgscharl@math.ucsb.edu}

\date{\today}

\begin{abstract}  A combinatorial proof of the Gordon Conjecture: The sum of two Heegaard splittings is stabilized if and only if one of the two summands is stabilized.
 \end{abstract}

\maketitle

\section{Introduction and basic background} 
 In 2004 the first author \cite{Q} presented a proof of the Gordon Conjecture, that the sum of two Heegaard splittings is stabilized if and only if one of the two summands is stabilized.  The same year, and a bit earlier, David Bachman \cite{Ba} presented a proof of a somewhat weaker version, in which it is assumed that the summand manifolds are both irreducible.  (A later version dropped that assumption.)  
 
The proofs in  \cite{Q} and \cite{Ba} are quite different.  The former is heavily combinatorial, essentially presenting an algorithm that will create, from a pair of stabilizing disks for the connected sum Heegaard splitting, an explicit pair of stabilizing disks for one of the summands. (Earlier partial results towards the conjecture, e. g. \cite{Ed} have been of this nature.)  In contrast, the proof in \cite{Ba} is a delicate existence proof, based on analyzing possible sequences of weak reductions of the connected sum splitting.   Both proofs have been difficult for topologists to absorb.  The present manuscript arose from the second author's efforts, following a visit to Dalian in 2007, to simplify and clarify the ideas in \cite{Q}.  (During that visit, MingXing Zhang was very helpful in providing the groundwork for this simplified version.)  

The most important strategic change here is an emphasis on symmetry.  In  \cite{Q} the roles of the two stabilizing disks on opposite sides of the summed Heegaard surface are quite different.  Here symmetry between the sides is maintained for as long as possible.  (In fact until Propositon \ref{prop:asymmetry}.) This adds a bit of complexity to the argument, but also some major efficiencies.  

The figures in this manuscript are meant to be viewed in color; readers confused by figures in a black-and-white version may find it helpful to look at an electronic version.

\bigskip

Since the argument easily extends to Heegaard splittings of bounded manifolds, for convenience we restrict to closed $3$-manifolds.  

A {\em Heegaard splitting} of a closed orientable $3$-manifold $M$ is a description of $M$ as the union of two handlebodies along their homeomorphic boundary.  That is $M = \cV \cup_S \cW$, where $\cV$ and $\cW$ are handlebodies and $S = \bdd \cV = \bdd \cW$.  The splitting is {\em stabilized} if there are properly embedded disks $V \subset \cV$ and $W \subset \cW$ so that $\bdd V \cap \bdd W$ is a single point in $S$.  

Suppose $M_+ = \cVp \cup_{S_+} \cWp$ and $M_- = \cVm \cup_{S_-} \cWm$ are two Heegaard split $3$-manifolds.  There is a natural way to obtain a Heegaard splitting $M = \cV \cup_S \cW$ for the connect sum $M = M_+ \# M_-$, where $S = S_+ \# S_-$:  Remove a $3$-ball $B^3_{\pm}$ from each of $M_{\pm}$, a ball that intersects $S_{\pm}$ in a single $2$-disk $D_{\pm}$.  Then attach $\bdd B^3_+$ to $\bdd B^3_-$ so that the disk $B_{\cV} = \bdd B^3_+ \cap V_+$ is identified to the disk $\bdd B^3_- \cap V_-$, the disk $B_{\cW} = \bdd B^3_+ \cap W_+$ is identified to the disk $\bdd B^3_- \cap W_-$ and so $\bdd D_+$ is identified to $\bdd D_-$ to create $S = S_+ \# S_-$.  This gives a Heegaard splitting $M = \cV \cup_S \cW$ with $\cV =  \cVp \natural_{B_{{\footnotesize \cV}}} \cVm$ and $\cW = \cWp \natural_{B_{{\footnotesize \cW}}} \cWm$.  See Figure \ref{fig:Overview2} (surfaces $P$ and $F$ to be explained later.)

    \begin{figure}[tbh]
    \centering
    \includegraphics[scale=0.6]{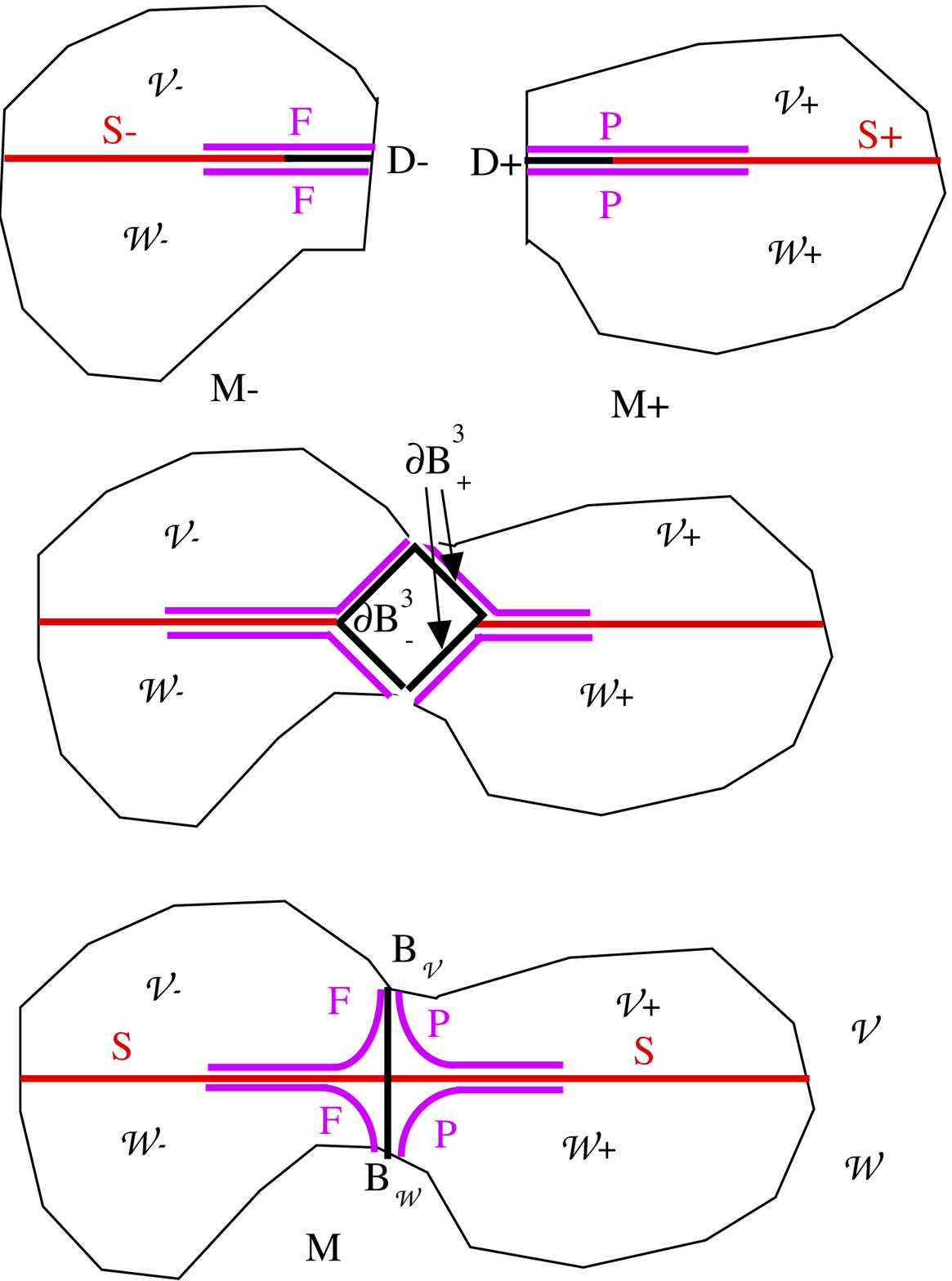}
    \caption{} \label{fig:Overview2}
    \end{figure}

In problem 3.91 of the Kirby problem \cite{Ki} list Gordon conjectured:

\begin{conj} $\cV \cup_S \cW$ is stabilized if and only if either $\cVp \cup_{S_+} \cWp$ or $\cVm \cup_{S_-} \cWm$ is stabilized.
\end{conj}

One direction of implication is obvious: a pair of stabilizing disks in $\cVp \cup_{S_+} \cWp$ or $\cVm \cup_{S_-} \cWm$ becomes a pair of stabilizing disks in $\cV \cup_S \cW$.  The interest is in the opposite direction.  

\section{The framework part 1: Rooted forests of disks in handlebodies}  

\begin{defin}  A {\em rooted tree} is a tree with a distinguished vertex called the root.  A {\em coherent numbering} of the vertices of a rooted tree is a numerical labelling of the vertices $\aaa_i, i \in \Np$ that increases along paths  that move away  from the root.  That is, if the path in the tree from the root to the vertex $\aaa_i$ passes through the vertex $\aaa_k$ (or if $\aaa_k$ is the root) then $k < i$.   

A {\em rooted forest} is a collection of rooted trees, one of which contains a distinguished root $\aaa_0$.  A {\em coherent numbering} of the vertices of a rooted forest is a numerical labelling of the vertices $\aaa_i, i \in \Np$ which restricts to a coherent numbering in each of the rooted trees.  
\end{defin}

Given an arbitrary forest with a distinguished root, it is easy to assign a coherent numbering: imagine the forest as a real forest in a hilly region with the distinguished root the lowest of all roots and the branches of all trees in the forest ascending upward.  Take a generic height function on the forest and assign numbers to each vertex in order of their height.  Feel free to skip some numbers; there is no requirement that the set of numbers assigned to vertices is contiguous in $\Np$.  Numbers that are assigned to vertices will be called {\em active} numbers.  

\bigskip

\noindent {\bf Examples:} A rooted tree with coherent numbering is clearly also a rooted forest  with coherent numbering.  Delete a vertex (other than the root) from a coherently numbered rooted tree and also delete all contiguous edges.  The result is a coherently numbered rooted forest $\cF$, with as many components as the valence of the vertex that is removed.  The root of each component of $\cF$ that does not contain the original root (now the distinguished root) is the vertex that was closest to the root in the original tree.  More generally, if $\cF$ is a coherently numbered rooted forest, and a vertex other than the distinguished root is removed, along with all contiguous edges, then the result is still a coherently numbered rooted forest, but with the number assigned to the vertex that has been removed now inactive.  

\begin{defin}   \label{defin:forests} Suppose $\cVp, \cVm$ is a pair of disjoint handlebodies, and $P \subset \bdd \cVp, F \subset \bdd \cVm$ are subsurfaces of their respective boundaries.  A {\em forest of disks}  (modeled on the rooted forest $\cF$) in the pair of handlebodies $\cVp, \cVm$ is a properly embedded collection of disks $\oV = \{ V_i \}$, one for each vertex $\aaa_i$ of $\cF$ so that:

\begin{enumerate}  
\item The disks alternate between lying in $\cVp$ and $\cVm$.  That is, suppose vertices $\aaa_i, \aaa_k$ are incident to the same edge in $\cF$.  Then $V_i \subset \cVp$ if and only if $V_k \subset \cVm$.

\item Suppose $\aaa_i$ is a vertex of $\cF$ and $V_i \subset \cVp$ (resp $V_i \subset \cVm$) is the corresponding disk.  If $\aaa_i$ is not a root, or is the distinguished root $\aaa_0$, there is a one-to-one correspondence between the edges of $\cF$ incident to $\aaa_i$ and arcs of $\bdd V_i \cap P$ (resp $\bdd V_i \cap F$).  If $\aaa_i$ is a non-distinguished root then there is one extra arc of $\bdd V_i \cap P$ (resp $\bdd V_i \cap F$) called the {\em root arc}.

\item Corresponding to each root arc in $\bdd V_i \cap P$ (resp $\bdd V_i \cap F$) there is a normally oriented pair of properly embedded arcs in $F$ (resp $P$) called {\em overpass arcs} (abbreviated {\em op-arcs}).  The op-arcs are all disjoint, both from each other and from $\bdd \oV$.  The collection of all op-arcs will be denoted $\cv$.  

The pairs of op-arcs will be required to have certain properties, which will be discussed below (see the end of Section \ref{sec:forests2} and Section \ref{sec:oparcs}).   
\end{enumerate}
\end{defin}

\noindent {\bf Seminal Example:}  Suppose the handlebody $\cV$ is expressed as the $\bdd$-connected sum of two handlebodies $\cVp$ and $\cVm$ along a disk $D$.  That is $\cV = \cVp \natural_D \cVm$. Consider a $\bdd$-reducing disk $V$ in $\cV$ and a distinguished point $x_0 \in \bdd V$.  It's easy to isotope $V$ rel $\bdd V$ so it intersects $D$ only in arcs.  Then the components of $V - \eta(D)$ are disks.

Here is a natural description of a tree $\cT$ embedded in the disk $V$:  For vertices, choose a point in the interior of each disk component of $V - D$.  For edges, choose, for each arc of $V \cap D$, an arc connecting the two vertices in the components of $V - D$ incident to that arc. Define the root of $\cT$ to be the vertex $\aaa_0$ that lies in the component of $V - D$ that has $x_0$ in its boundary.  Then the components of $V - \eta(D)$ constitute a tree of disks in $\cVp \cup \cVm$, modeled on $\cT$, with $P$ the copy of $D$ in $\bdd \cVp$ and $F$ the copy of $D$ in $\bdd \cVm$.  Since the only root is the distinguished root, there are no op-arcs.

    \begin{figure}[tbh]
    \centering
    \includegraphics[scale=0.6]{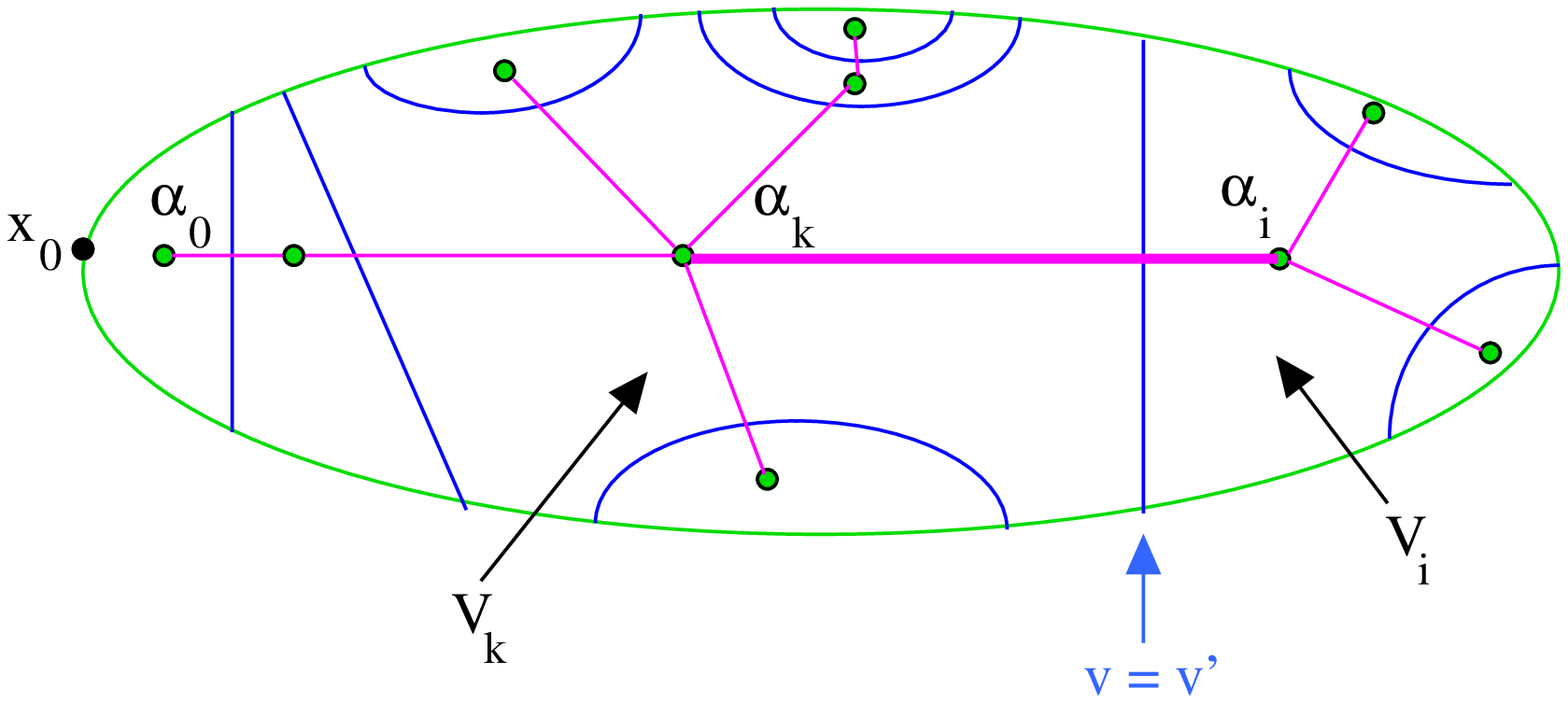}
    \caption{} \label{fig:Tree}
    \end{figure}

\bigskip

This example, though seminal to our discussion, is deceptive in two ways:  First, in this example the surfaces $P \subset \bdd \cVp$ and $F \subset \bdd \cVm$ are simply two sides of the same surface (namely $D$) and so can be naturally identified.  In general this will not be true.  Second, and most deceptively, an edge in the tree $\cT$ between two vertices, say $\aaa_i$ (representing $V_i \subset \cVm$) and $\aaa_k$ (representing $V_k \subset \cVp$)  corresponds, as required, both to a component $v$ of $\bdd V_i \cap F$ and $v'$ of $\bdd V_k \cap P$.  But what is true here and will not be true in general, is that both $v$ and $v'$ can be thought of as {\em the same} arc, namely a single component of $V \cap D$.  In general (only in part because there will be no natural identification of $P$ and $F$) the two arcs $v \subset \bdd V_i$ and $v' \subset \bdd V_k$ determined by a single edge in $\cT$ will, at least {\it prima facie}, have nothing to do with each other.  

\bigskip

\noindent {\bf Labeling convention:}   There is an efficient way to label the properly embedded arcs in $F$ and $P$ that come from a forest $\oV$ of disks in $\cVp, \cVm$ that is modeled on a coherently numbered forest $\cF$.  

First note that there is a natural way to assign a unique label to each edge in the forest $\cF$, namely give each edge the label of the vertex at its end that is most distant from the root.  That is, if the edge in $\cF$ has ends at vertices $\aaa_i$ and $\aaa_k$, with $\aaa_k$ either closer to the root or perhaps the root itself, so $k < i$, then label the edge $e_i$.  

As discussed in the example above, each edge $e_i$ in $\cF$ actually represents two arcs since $e_i$ is incident to two vertices $e_k$ and $e_i$ in $\cF$.  One arc is in $\bdd \oV \cap P \subset \bdd \cVp$ and the other is an arc in $\bdd \oV \cap F \subset \bdd \cVm$.  If, say, $V_i \subset \cVp$, so $V_k \subset \cVm$ then one end of $e_i$ corresponds, under Definition \ref{defin:forests}, to an arc of $\bdd V_i \cap P$, and the other end of $e_i$ corresponds to an arc of $\bdd V_k \cap F$.  It is natural to call these arcs $v_i^+$ and $v_i^-$ respectively, though it is perhaps counterintuitive that with this convention, $v_i^- \subset \bdd V_k$.   Symmetrically, if $V_i \subset \cVm$, so $V_k \subset \cVp$ then the arc of $\bdd V_i \cap F$ corresponding to the end of $e_i$ at $\aaa_i$ is called $v_i^-$, and the arc of $\bdd V_k \cap P$ corresponding to the end of $e_i$ at $\aaa_k$ is called $v_i^+$.

Now extend this labeling in the natural way to the root arcs and op-arcs:   
If $V_i \subset \cVp$ (resp $\cVm$) is a non-distinguished root, label the root arc $v_i^+ \subset P$ (resp $v_i^- \subset F$).  Label the corresponding {\em pair} of op-arcs in $F$ (resp $P$) by $v_i^-$ (resp $v_i^+$).  See Figure \ref{fig:Tree2} for how this labels arcs in the Seminal Example and Figure \ref{fig:Labeling2} for how the labelling may appear on $\bdd \cVp \cup \bdd \cVm$ in the more general case.  

    \begin{figure}[tbh]
    \centering
    \includegraphics[scale=0.6]{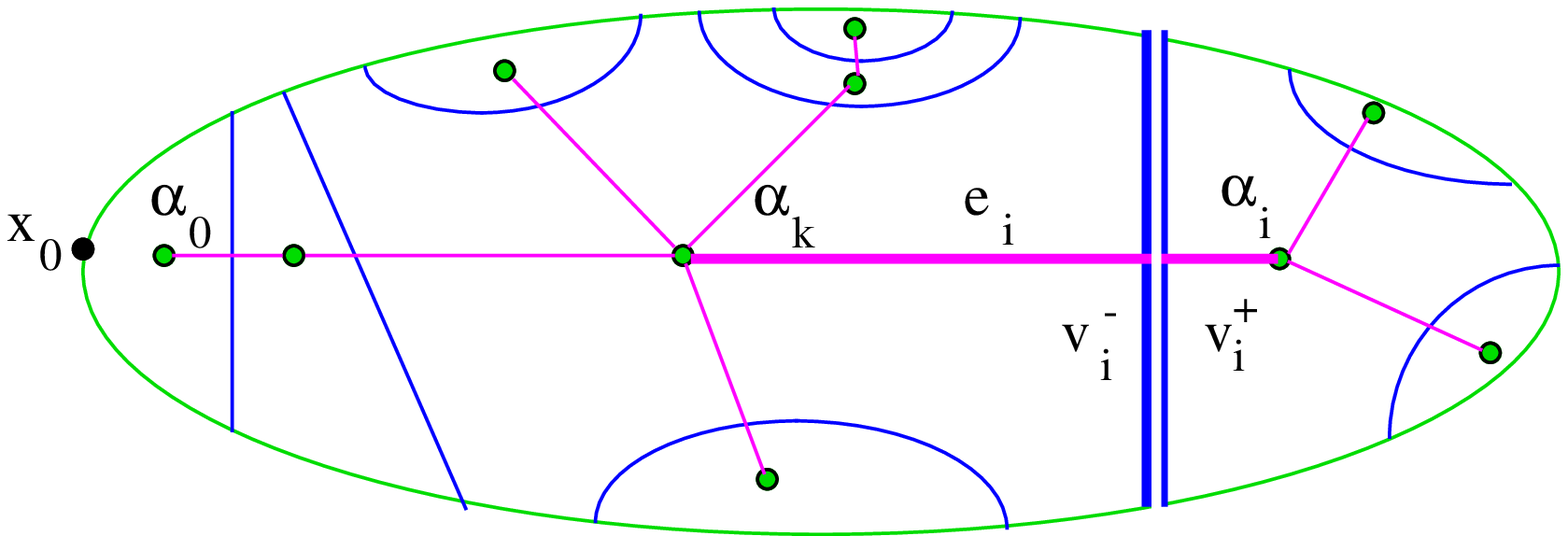}
    \caption{} \label{fig:Tree2}
    \end{figure}

    \begin{figure}[tbh]
    \centering
    \includegraphics[scale=0.6]{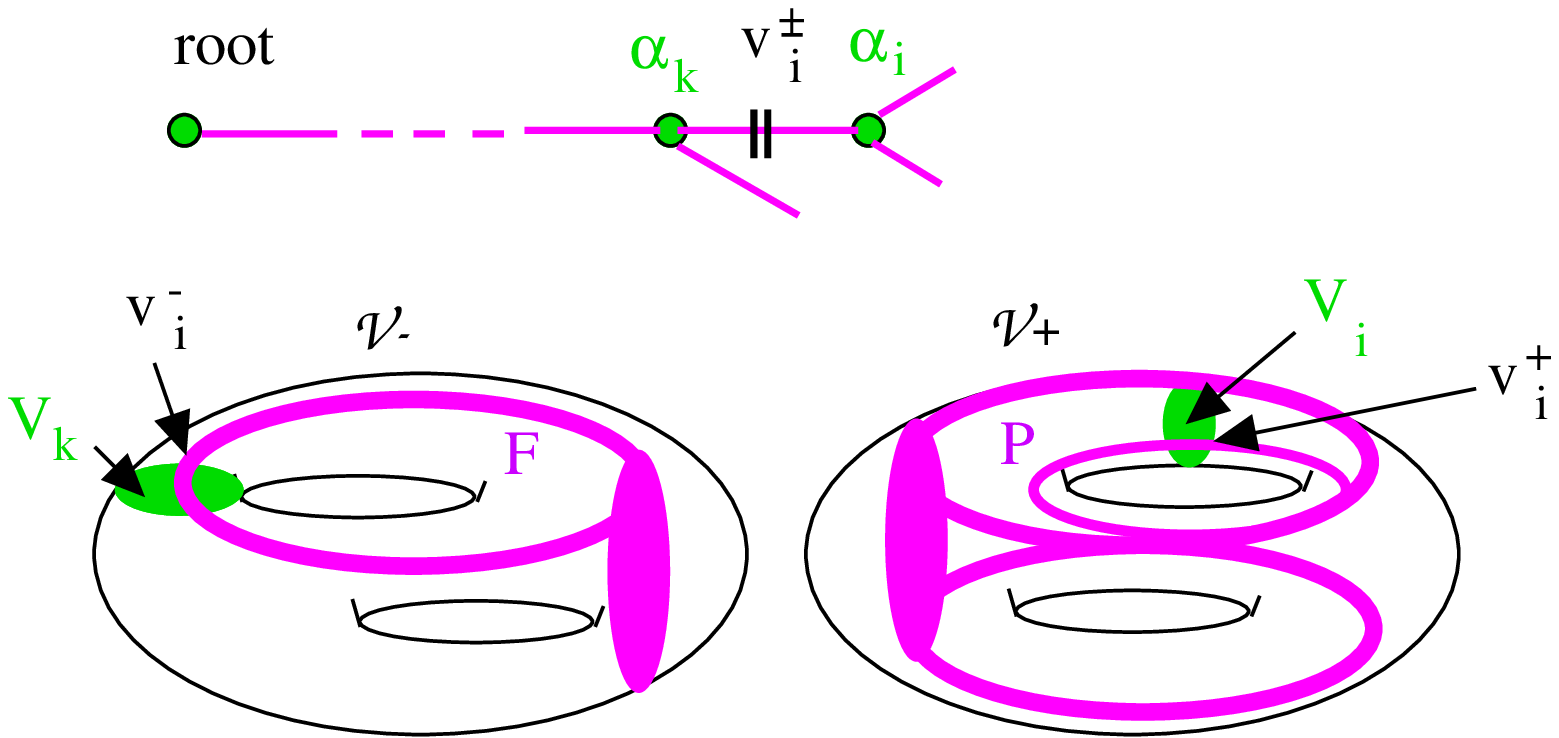}
    \caption{} \label{fig:Labeling2}
    \end{figure}
    
\section{The framework part 2: Stabilizing forests of disks in a Heegaard splitting}   \label{sec:forests2}

We now extend this construction to a pair of Heegaard split $3$-manifolds:

\begin{defin}   Suppose $M_+ = \cVp \cup_{S_+} \cWp$, $M_- = \cVm \cup_{S_-} \cWm$ are two closed orientable $3$-manifolds, and $P \subset S_+$ and $F \subset S_-$ are subsurfaces.  Suppose $\oV = \{ V_i \}$ and $\oW = \{ W_j \}$  together with associated op-arcs $ \cv, \cw$ are forests of disks in the pairs $\cVp, \cVm$ and $\cWp, \cWm$ respectively.  Let $\{ v_i^{\pm} \}, \{ w_j^{\pm} \}$ be the collections of arcs  labeled as described above. (Some of these are root arcs, some of them pairs of op-arcs; typically they are arcs in $(\bdd \oV \cup \bdd \oW) \cap P$ and $(\bdd \oV \cup \bdd \oW) \cap F$).  

Suppose that only a single point $x_0 \in \bdd \oV \cap \bdd \oW \subset S_+ \cup S_-$ lies outside of $F \cup P$, a point that lies in $\bdd V_0 \cap \bdd W_0$.  Suppose further that all op-arcs among the $\{ v_i^{\pm} \}$ (denoted $\cv$) are disjoint from all op-arcs among the $\{ w_j^{\pm} \}$ (denoted $\cw$).  Then $\oV, \oW$ together with associated op-arcs $ \cv, \cw$ is a {\em stabilizing pair} of forests of disks for $\cVp \cup_{S_+} \cWp$ and $\cVm \cup_{S_-} \cWm$.
\end{defin}

\noindent {\bf Seminal example:}  Suppose $M_+ = \cVp \cup_{S_+} \cWp$ and $M_- = \cVm \cup_{S_-} \cWm$ are two Heegaard split $3$-manifolds and $M = \cV \cup_S \cW$ is the connected sum splitting on $M = M_+ \# M_-$, where $S = S_+ \# S_-$.  Suppose that $\cV \cup_S \cW$ is a stabilized splitting.  Then there are disks $V \subset \cV$ and $W \subset \cW$ so that $\bdd V \cap \bdd W$ is a single point  $x_0 \in S$.  Following  the Seminal Example for handlebodies above, $V$ and $W$ give rise to rooted trees of disks $\oV$ and $\oW$ in the pairs $\cVp, \cVm$ and $\cWp, \cWm$ respectively.  These rooted trees, having no non-distinguished roots, also have no op-arcs $\cv$ or $\cw$. The original Heegaard splittings for $M_{\pm}$ are obtained from this picture of $\cVm \cup_{(S_- - B_-)} \cWm$ and $\cVp \cup_{(S_+ - B_+)} \cWp$ by identifying the disks $B_{{\footnotesize \cV}}$ with $B_{{\footnotesize \cW}}$ in both manifolds.  The resulting disk in $S_-$ we regard as $F$ and the resulting disk in $S_+$ we regard as $P$.  Except for $x_0$, all intersections between $\oV$ and $\oW$ lie where the disks $B_{\cV}$ and $B_{\cW}$ have been identified, namely in the disk $F \subset S_-$ and the disk $P \subset S_+$.  Hence $\oV$ and $\oW$ constitute a stabilizing pair of forests for the pair of Heegaard split manifolds $M_+ = \cVp \cup_{S_+} \cWp$ and $M_- = \cVm \cup_{S_-} \cWm$.

\bigskip

In this example, there is a clear connection between how the boundaries of the disks $\oV$ and $\oW$ intersect in $S_+$ and how they intersect in $S_-$.  Consider a pair of arcs $v_i^+$ and $w_j^+$ in the disk $P$, arcs isotoped rel their boundary points in $\bdd P$ to intersect minimally.  Then the arcs intersect (in precisely one point ) if and only if the pair of points $\bdd v_i^+$ separate the pair of points $\bdd w_j^+$ in the circle $\bdd P$.  Since, in this example, $v_i^+$ and  $v_i^-$ are copies of the same arc of $V \cap B_{\cV}$ (and, symmetrically,  $w_j^+$ and  $w_j^-$ are copies of the same arc of $W \cap B_{\cW}$), the pair of points $\bdd v_i^+$ separate the pair of points $\bdd w_j^+$ in the circle $\bdd P = \bdd D_+$ if and only if the pair of points $\bdd v_i^-$ separate the pair of points $\bdd w_j^-$ in the circle $\bdd F = \bdd D_-$.  To summarize, $|v_i^+ \cap w_j^+| = |v_i^- \cap w_j^-| \leq 1$.

The relation between $|v_i^+ \cap w_j^+|$ and $ |v_i^- \cap w_j^-|$ is more complicated in the general case.  To begin with, as mentioned above, the arcs $v_+ \subset P$ and $v_- \subset F$ may not have anything to do with each other.  Moreover, since $P$ (resp $F$) is an arbitrary subsurface of $S_+$ (resp $S_-$), two proper arcs, even when isotoped rel boundary to intersect minimally, may still intersect in a large number of points.   

Complicating things further, one of $v_i^{\pm}$ (or $w_j^{\pm}$) may represent a pair of op-arcs, about which so far we've said only this:  Each pair of op-arcs, say $v_i^+ = v_i^a \cup v_i^b \subset P$, is normally oriented, disjoint from all other arcs $v_k^+ \subset P$ and also disjoint from all pairs of op-arcs $w_j^+ \in \cw \subset P$.  We now introduce two properties which describe how such a pair of op-arcs $v_i^+$ is assumed to intersect the remaining arcs $\bdd \oW \cap P$, the arcs $w_j^+$ that are not themselves op-arcs.  Symmetric statements apply to pairs of op-arcs $v_k^- \subset F$ and pairs of op-arcs $w_k^+ \subset P$ and $w_k^- \subset F$.

\bigskip

Near $v_i^+ = v_i^a \cup v_i^b$ in $P$, call the side of $v_i^+$ towards which the normal orientation of $v_i^+$ points the {\em inside of $v_i^+$} and the other side the {\em outside} of $v_i^+$.

\bigskip

{\bf Separation Property of op-arcs:}   Suppose $v_i^+ = v_i^a \cup v_i^b$ is a pair of op-arcs in $P$. Then for any arc $w_j^+$, each component of $w_j^+ - v_i^+$ has ends
\begin{enumerate}
\item both incident to the outside of $v_i^+$ or
\item both incident to $\bdd P$ (when $v_i^+$ and $w_j^+$ are disjoint) or
\item one incident to $\bdd P$ and one incident to the outside of $v_i^+$ or
\item one incident to the inside of $v_i^a$ and one incident to the inside of $v_i^b$
\end{enumerate}
Subarcs of $w_j^+$ of the last type are said to be {\em on the overpass} associated to $v_i^+$ and, in analogy to railroad ties, are called {\em op-ties} for the overpass.  See Figure \ref{fig:Parallelism}.  Components of $w_j^+ - v_j^+$ of the first three types are said to be {\em off the overpass} associated to $v_i^+$.

\bigskip

 {\bf Parallelism Property of op-ties:}   Suppose $v_i^+ = v_i^a \cup v_i^b$ is a pair of op-arcs in $P$. Then all op-ties for the overpass associated to $v_i^+$ are parallel.  To be explicit:  suppose $\aaa$ and $\aaa'$ are two components of $\bdd \oW - v_i^+$ and each has one end incident to the inside of $v_i^a$ and the other incident to the inside of $v_i^b$.  Then the rectangle in $P$ formed by the union of $\aaa, \aaa'$ and subarcs of $v_i^a$ and $v_i^b$ bounds a disk in $P$.
See Figure \ref{fig:Parallelism}.

    \begin{figure}[tbh]
    \centering
    \includegraphics[scale=0.6]{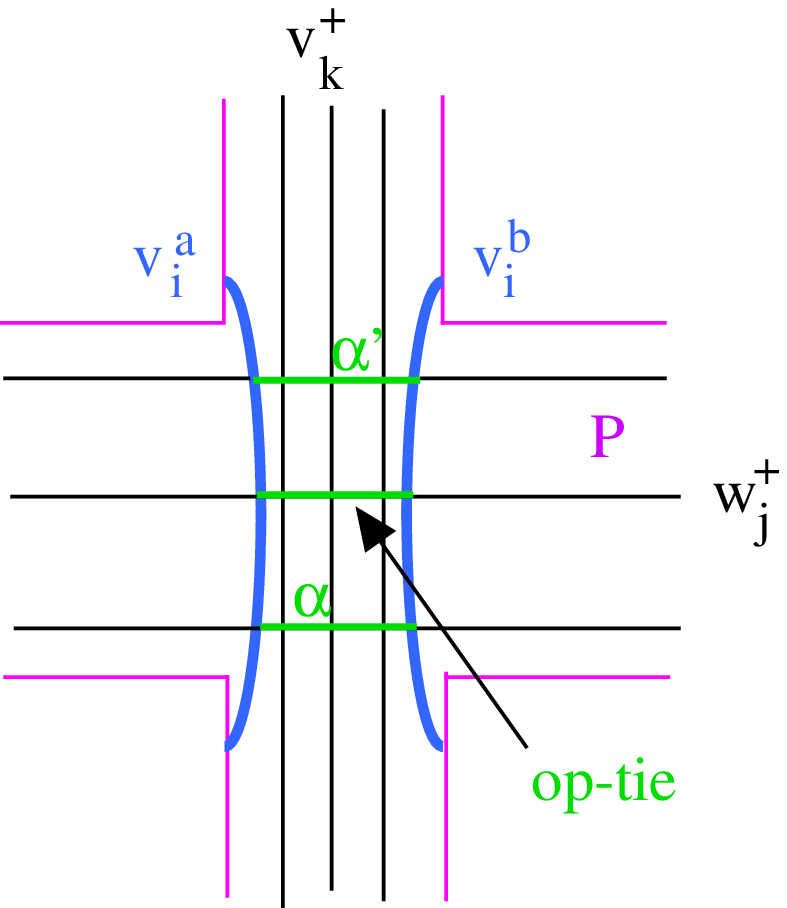}
    \caption{} \label{fig:Parallelism}
    \end{figure}

\bigskip

Two further properties of op-arcs that we assume will be given later. (See Section \ref{sec:oparcs}.)  For now, we only introduce a useful definition:

\begin{defin} \label{defin:ground}  A component of $v_i - \cw$ that lies off of every $\cw$ overpass is called a {\em ground arc} in $v_i$.  Symmetrically, a component of $w_j - \cv$ that lies off of every $\cv$ overpass is called a {\em ground arc} in $w_j$.
\end{defin}

Note that since the op-arcs $\cv$ and $\cw$ are disjoint, each overpass arc is, perhaps counterintuitively, a ground arc.

\section{The first pairings $\rho_{\pm}$}

We have seen that the Separation Property guarantees that for each arc $w_j^+$ and pair of op-arcs $v_i^+ = v_i^a \cup v_i^b$, $|v_i^+ \cap w_j^+| = 2|v_i^a \cap w_j^+|  = 2|v_i^b \cap w_j^+|$.  With this in mind, the following is a natural definition.
   
\begin{defin}  Let $M_+ = \cVp \cup_{S_+} \cWp$, $M_- = \cVm \cup_{S_-} \cWm$, $P \subset S_+$ and $F \subset S_-$ be as above.  Suppose the families of disks $\oV, \oW$ and associated op-arcs $\cv, \cw$ is a stabilizing pair of coherently numbered forests for the pair of Heegaard splittings.  Define two pairings $\rho_{\pm}: \No \times \No \to \Np$ by \begin{itemize} 
\item $\rho_{\pm}(i, j) = |v_i^{\pm}  \cap w_j^{\pm}|$ when neither $v_i^{\pm}$ nor $w_j^{\pm}$ is a pair of op-arcs or
\item $\rho_{\pm}(i, j) = |v_i^{\pm}  \cap w_j^{\pm}|/2$ when either $v_i^{\pm}$ or $w_j^{\pm}$ is a pair of op-arcs
\item $\rho_{\pm}(i, j) = 0$ if $v_i^{\pm}$ or $w_j^{\pm}$ is not defined, e. g.  if $i$  (resp $j$) is not among the indices of the disks in the rooted forest $\oV$ (resp. $\oW$).  That is, when $i$ (resp $j$) is an inactive index.
\end{itemize}
\end{defin}

Explanatory notes:  Here $|v_i^+  \cap w_j^+ |$ (resp $|v_i^-  \cap w_j^- |$) means the number of intersection points, minimized by isotopy rel boundary, of the two arcs $v_i^+$ and $w_j^+$ in $P$ (resp  $v_i^-  \cap w_j^-$ in $F$). See Figure \ref{fig:Matrixrho}.   If  $v_i^+$ is a pair of op-arcs in $P$ then we have seen that $\rho_{+}(i, j)$ is the number of intersections of $w_j^+$ with either one $v_i^a$ or $v_i^b$ of the op-arc pair $v_i^+$ (and symmetrically for a pair of op-arcs $v_i^-$ in $F$ or a pair of op-arcs $w_j^{+} \subset P$ or $w_j^{-} \subset F$). 

 \begin{figure}[tbh]
    \centering
    \includegraphics[scale=0.6]{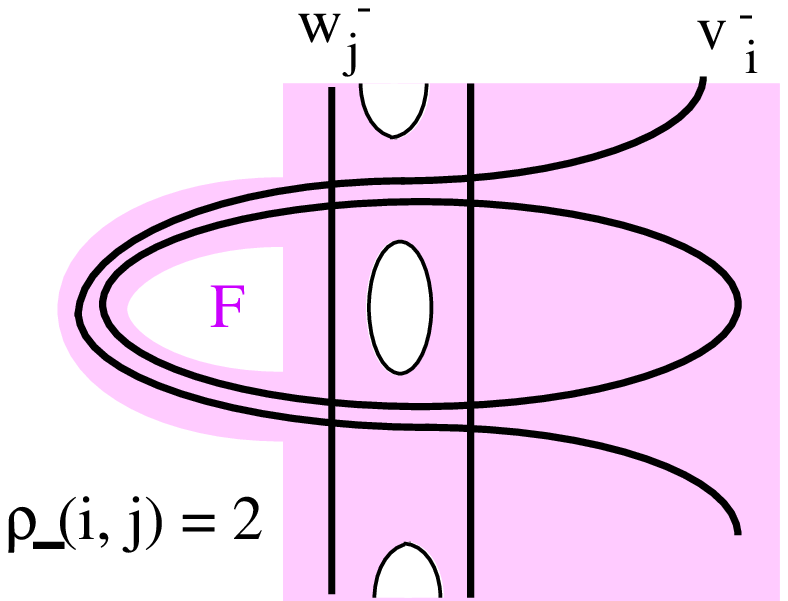}
    \caption{} \label{fig:Matrixrho}
    \end{figure}

\begin{defin}  Suppose $M_+ = \cVp \cup_{S_+} \cWp$ and $M_- = \cVm \cup_{S_-} \cWm$; $P \subset S_+$ and $F \subset S_-$; and forests of disks $\oV, \oW$ and associated op-arcs $\cv, \cw$ are all given as above. A pair $(i, j)$ is {\em peripheral} if for all $(i', j') \neq (i, j)$ with $i' \geq i$ and $j' \geq j$, $\rho_+(i', j') = \rho_-(i', j') =  0$. 
\end{defin}

\begin{figure}[tbh]
    \centering
    \includegraphics[scale=0.6]{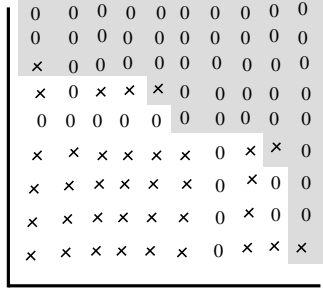}
    \caption{Grey shows peripheral lattice points $(i, j)$; $\times$ means a non-zero entry}
     \label{fig:rhotable1}
    \end{figure}

\begin{lemma}  \label{lemma:terminate} Suppose $(i, j)$ is peripheral and
\begin{itemize} 
\item $\rho_+(i, j) = 1$
\item $V_i \subset \cVp$ 
\item $W_j \subset \cWp$
\end{itemize}
Then $M_+ = \cVp \cup_{S_+} \cWp$ is a stabilized splitting.

Symmetrically,  if $\rho_-(i, j) = 1$, $V_i \subset \cVm$ and $W_j \subset \cWm$, then $M_- = \cVm \cup_{S_-} \cWm$ is a stabilized splitting.
\end{lemma}

\begin{proof}  $V_i \subset \cVp$ and $W_j \subset \cWp$ will be the stabilizing disks.  By the labeling convention, $\bdd V_i \cap P$ consists of arcs $v_i^+$ and (possibly) other arcs $v_{i'}^+, i' > i$.  Similarly, $\bdd W_j \cap P$ consists of the arc $w_j^+$ and (possibly) other arcs $w_{j'}^+, j' >j$.   Since $(i, j)$ is peripheral, each $v_{i'}^+$ with $i' > i$ is disjoint from $\bdd W_j$ and each $w_{j'}^+$ with $j' > j$ is disjoint from $\bdd V_i$.  Hence the only points in $\bdd V_i \cap \bdd W_j$ are those in $v_i^+ \cap w_j^+$.  Since $\rho_+(i, j) = 1$, there is exactly one such point.  Hence $\bdd V_i \cap \bdd W_j$ is a single point, so $\cVp \cup_{S_+} \cWp$ is a stabilized splitting. \end{proof}

\section{Further properties of the op-arcs} \label{sec:oparcs}

We now introduce two further properties which pairs of op-arcs are assumed to satisfy.  Since there are no op-arcs in the Seminal Example above, these properties are vacuously satisfied in that example.  Part of the argument will be to show that the fundamental construction described below preserves all these properties of pairs of op-arcs.  This section describes properties of op-arcs $v_i^+$ in $P$; symmetric statements are true for op-arcs $w_j^+$ in $P$ and op-arcs 
$v_i^-$ and $w_j^-$ in $F$.  It may be helpful, when $v_i^+$ is specifically meant to be a pair of op-arcs, to denote it $\cv_i$ and when $w_j^+$ is meant to be a pair of op-arcs, denote it by $\cw_j$.

\bigskip

{\bf Ordering Property for op-ties:} Suppose $\aaa \subset \bdd \oW$ is an op-tie for the pair of op-arcs $\cv_i$.  For any $k \geq i$, $v_k^+$ is disjoint from the interior of $\aaa$.  

In particular, suppose $\cv_k = v_k^+, k > i$ is also a pair of op-arcs $\cv_k$, and $\aaa'$ is an op-tie for the overpass associated with $\cv_k$ with $\aaa \cap \aaa' \neq \emptyset$.  Then neither end of the arc $\aaa \cap \aaa'$ can lie on $\cv_k$, so both ends lie on $\cv_i$ and $\aaa \subset \aaa'$.  See Figure \ref{fig:Ordering}.

 \begin{figure}[tbh]
    \centering
    \includegraphics[scale=0.6]{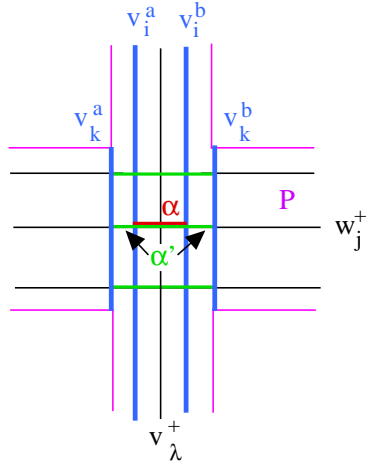}
    \caption{Here the Ordering Property implies $\lambda < i < k.$} \label{fig:Ordering}
    \end{figure}

\bigskip

The last property of op-arc pairs, stated below, requires some background:  Here is a way of using the pairs of op-arcs in $P$ to construct a new surface $\hP$.  This construction will be called {\em building the overpasses}.  It is parallel to the idea of an ``abstract tree" for $F$ and $P$, found in \cite{Q}.  It's important to understand that, as with the abstract tree, the building of overpasses is done in the abstract, as a way to express a property of the op-arcs, and is {\em not} a construction actually performed inside of $M_+$ or $M_-$. Here we describe how to build a single overpass, one associated to the  pair of op-arcs $\cv_i  = \cv_i^a \cup \cv_i^b$.  First cut $P$ along $\cv_i$.  The resulting surface $P'$ has two copies of $\cv_i^a$ and two copies of $\cv_i^b$ in its boundary.  One copy of $\cv_i^a$ in $\bdd P'$ is incident to the outside of $\cv_i^a$ in $P$ and one copy of $\cv_i^b$ is incident to the outside of $\cv_i^b$ in $P$.  Identify these two arcs in $\bdd P'$ and call the resulting arc $v''^+_i$ and surface $P''$.  The other copies of  $\cv_i^a$ and  $\cv_i^b$ remain in $\bdd P''$.  

Building the overpass as described does nothing particularly interesting to the other arcs $v_k$, since these are disjoint from $\cv_i$.  The arcs $\{ w_j^+ \}$ that intersect $\cv_i$ are cut up when the overpass is built:  Each arc in $\{ w_j^+ \}$ naturally gives rise to perhaps many properly embedded arcs in $P'$ (each $w_j^+$ is cut into pieces by $\cv_i$) and, less obviously, to a single special arc in $P''$:  By the Parallelism Property, the ends of $w_j^+ - \cv_i$ lying just outside $\cv_i^a$ match naturally with the ends of $w_j^+ - \cv_i$ lying just outside $\cv_i^b$ and so can be attached in $P''$ to become a proper arc $w''^+_j$ in $P''$.  A simple picture of the special arc $w''^+_j$ is that it is the arc obtained from $w_j^+$ by collapsing all the op-ties of $w_j^+$ that lie on the overpass associated to $\cv_i$.  The upshot is that, in $P''$, $w_j^+$ is fractured into a collection of op-ties, each now a proper arc in $P''$ and no longer indexed, plus a single arc $w''^+_j$ that is the end-point union of all subarcs of $w_j^-$ that lie off the overpass.

 \begin{figure}[tbh]
    \centering
    \includegraphics[scale=0.6]{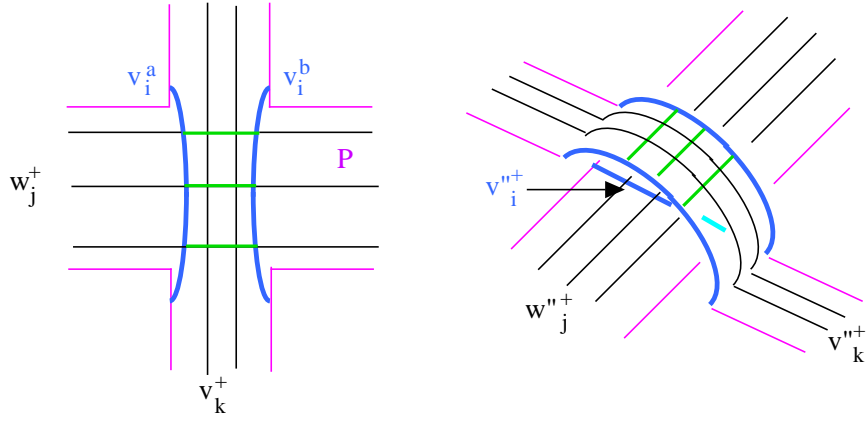}
    \caption{Building an overpass.  Green arcs on right are unindexed} \label{fig:Building}
    \end{figure}

Define $\hP$ to be the surface obtained from $P$ by building all the overpasses at once.  That is, perform the operation just described on all pairs of op-arcs $\cv$ and $\cw$ simultaneously.  There is no ambiguity in the construction, since $\cv$ and $\cw$ are assumed to be disjoint.  It may be worth noting (but is not important to the argument) that, following the Ordering Property above, when an op-tie $\aaa$ for $\cv_i$ overlaps with an op-tie $\aaa'$ for $\cv_k, k > i$ then $\aaa \subset \aaa'$.  So in the construction of $\hP$,  $\aaa'$ is fractured into pieces by $\cv_i$ and the proper arc in $\hP$ corresponding to $\aaa'$ is obtained by assembling those pieces that lie off the overpasses $\cv_i$ (and whatever other overpasses may pass through $\aaa'$).

\bigskip

Our final and most delicate assumption on op-arcs is:

\bigskip

{\bf  Disk Property:} The surface $\hP$ is a disk.

\bigskip

In the disk $\hP$ each curve $v_i^+$ and each curve $w_j^+$ may be fractured into many pieces.  One component constructed out of $w_j^+$, for example, will be the end point union of all subarcs of $w_j^+$ that lie off of every overpass, i. e. the ground arcs of $w_j^+$.  (See Definition \ref{defin:ground}.)  This arc   in $\hP$ corresponding to $w_j^+$ is denoted $\hat{w}_j^+$ (and appears as $w"^+_j$ in Figure \ref{fig:Building}, where only one overpass is built).   Another component constructed out of $w_j^+$ might be the end point union of all subarcs of $w_j^+$ that lie on the overpass determined by $\cv_k$ and off all overpasses determined by any $\cv_i, i < k$.  We have no notation for such arcs, since arcs in $\hP$ coming from op-ties will play no role in the argument.   Among op-arcs, the pair of arcs $\cv_i$  (resp $\cw_j$) in $P$ becomes a single proper subarc of $\hP$ which we denote $\hat{v}_i^+$ (resp. $\hat{w}_j^+$). (The curve $\hat{v}_i^+$ appears as $v"^+_i$ in Figure \ref{fig:Building}.)  The union of all such curves (coming from $\cv$ and $\cw$) in $\hP$ will be denoted $\hat{\cv}$ and $\hat{\cw}$ respectively.  Of course they are no longer op-arcs in $\hP$ because all overpasses have been built there.

\section{The pairing $\sss$ of arcs in $\hP$}

\begin{lemma} \label{lemma:efficient}  Any two arcs $\hv_i^+$ and $\hw_j^+$ (resp. $\hv_i^-$ and $\hw_j^-$) intersect efficiently in $\hP$ (resp $\hF$).  That is,  $|\hv_i^+ \cap \hw_j^+|$ cannot be reduced by isotopies of $\hv_i^+$ and $\hw_j^+$ in $\hP$ rel $\bdd \hP$.
\end{lemma}

\begin{proof} We must show that no complementary component of the two curves in $\hP$ is a bigon, that is, a disk bounded by the union of a subarc of $\hv_i^+$ and a subarc of $\hw_j^+$.  Suppose, towards a contradiction, that there were such a bigon $B$.  Since $\cv$ and $\cw$, hence $\hat{\cv}$ and $\hat{\cw}$, are disjoint, at least one side of the bigon, say the side on $\hv_i^+$ does not come from an op-arc.

Consider first the case in which the interior of $B$ is disjoint from all curves $\hcv \cup \hcw$ coming from op-arcs.  Then $B$ would also lie in $P$ since the interior of $B$ is disjoint from  the curves $\hcv \cup \hcw$ along which $P$ was cut and glued.  But the conclusion $B \subset P$ would violate our initial assumption that the curves $v_i^+ \subset P$ and $w_j^+ \subset P$ intersect efficiently in $P$.

Now suppose that the interior of $B$ is not disjoint from $\hcv \cup \hcw$.  Since $\hv_i^+$ intersects only op-arcs in $\hcw$ and $\hw_j^+$ intersects only op-arcs in $\hcv$, any component of $\hcw \cap B$ (resp $\hcv \cap B$) would have both ends on $\hv_i^+$ (resp $\hw_j^+$).  Hence there is a bigon $B' \subset B$ between a subarc of $\hv_i^+$ (say) and a subarc of some $\hcw_k^+ \in \hcw$.  Moreover, if $B'$ is chosen to be an innermost such example, then the interior of $B'$ would be disjoint from $\hcv \cup \hcw$.  Then, just as above, $B'$ would lie entirely in $P$ and this would violate the initial assumption that the arcs $v_i^+$ and $w_k^+$ intersect efficiently in $P$.
\end{proof}

Following Lemma \ref{lemma:efficient}, the Disk Property leads naturally to a new pairing:

\begin{defin}  Analogous to the intersection pairings $\rho_{\pm}$ in $P$ and $F$ define intersection pairings $\sss_{\pm}: \No \times \No \to \Np$ in the disks $\hP$ and $\hF$ by 
\begin{itemize} 
\item $\sss_{\pm}(i, j) = |\hv_i^{\pm}  \cap \hw_j^{\pm}|$ or
\item $\sss_{\pm}(i, j) = 0$ if $\hv_i^{\pm}$ or $\hw_j^{\pm}$ is not defined. (That is, if $i$ or $j$ is an inactive index.)
\end{itemize}
\end{defin}

\begin{lemma}  \label{lemma:sigmaone} For each $(i, j)$
\begin{enumerate}
\item $\sss_{\pm}(i, j)  \leq \rho_{\pm}(i, j)$ and 
\item $\sss_{\pm}(i, j)  \leq 1$
\end{enumerate}
\end{lemma}

\begin{proof}  As usual, we focus on $\sss_+$ defined on arcs in $P$; the case for $\sss_-$ defined on arcs in $F$ is symmetric.

 For the first claim, note that any intersection point of $\hv_i^{+}$ with $\hw_j^{+}$ in $P$ is merely a particular type of intersection point of $v_i^+$ with $w_j^+$, namely one which is not on any overpass.

The second claim follows immediately from the fact that $\hP$ is a disk and, following Lemma \ref{lemma:efficient}, the arcs $\hv_i^+$ and $\hw_j^+$ intersect efficiently in $\hP$. \end{proof}

\begin{cor} \label{cor:periph} If $(i, j)$ is peripheral, then $\rho_{\pm}(i, j) = \sss_{\pm}(i, j) \leq 1$.
\end{cor}

\begin{proof}  Following Lemma \ref{lemma:sigmaone}, the statement is obvious if $\rho_{\pm}(i, j) = 0$.  Suppose, say, $\rho_{+}(i, j) > 0$, and $x \in v_i^+ \cap w_j^+$.  If $x$ were in the interior of any op-tie in $w_j^+$, coming from a pair of op-arcs $\cv_k$, say, then it would follow from the Ordering Property that $k > i$. Then the ends of the op-tie would be points in $v_k^+ \cap w_j^+$, contradicting the fact that $(i, j)$ is peripheral.  Hence $x$ lies on no overpass associated with any of the $\cv$.  Symmetrically, it's on no overpass associated with any of the $\cw$.  Hence $x \in \hv_i^+ \cap \hw_j^+ \subset \hP$.  

Summarizing, this shows that for any peripheral $(i, j)$,  $\sss(i, j)_+ \geq \rho(i,j)_+$.  The result then follows from Lemma \ref{lemma:sigmaone}.
\end{proof}  

\begin{defin}  Suppose $M_+ = \cVp \cup_{S_+} \cWp$, $M_- = \cVm \cup_{S_-} \cWm$ and surfaces $P \subset S_+$ and $F \subset S_-$ are given as above and disks $\oV, \oW$ and associated op-arcs $\cv, \cw$ is a stabilizing pair of coherently numbered forests for the pair of Heegaard splittings.  Then the forests are {\em coordinated} if for all $(i, j) \in \No \times \No$, $\sss_+(i, j) = \sss_-(i, j)$.
\end{defin}

\noindent {\bf Seminal Example:} For the Seminal Example, it was observed that for all $(i, j) \in \No \times \No$, $\rho_+(i, j) = \rho_-(i, j)$.  But in that example there are no op-edges, so $\hP = P, \hF = F$.  Then for all  $(i, j)$, $\sss_{+}(i, j) = \rho_{+}(i, j) = \rho_{-}(i, j) = \sss_{-}(i, j)$.  Hence the forests of disks in the Seminal Example are coordinated.

\section{A digression on some operations on curves and surfaces} \label{sec:surfaces}

Suppose $A$ is an annulus containing a core circle $c$ and two spanning arcs $e$ and $w$.   Suppose $\lambda_w$ is a proper arc in $A$ that intersects $w$ once and is disjoint from $c$ and $e$.  Then there is an arc $\lambda_e$ in $A$, unique up to isotopy rel $\bdd$, that has the same ends as $\lambda_w$ but is disjoint from $c$ and $w$ and intersects $e$.  One way of describing how $\lambda_e$ is derived from $\lambda_w$ is to band-sum $\lambda_w$ to $c$ along $w$.  The same is true if $\lambda_w$ consists of a disjoint family of arcs in $A$, each component of which intersects $w$ in a single point and is disjoint from $w$ and $e$.  The change could be described as band-summing $\lambda_w$ along $w$ to $c$; as many copies of $c$ are band-summed as there are components of $\lambda_w$.  See Figure \ref{fig:Westoeast}.
 
  \begin{figure}[tbh]
    \centering
    \includegraphics[scale=0.6]{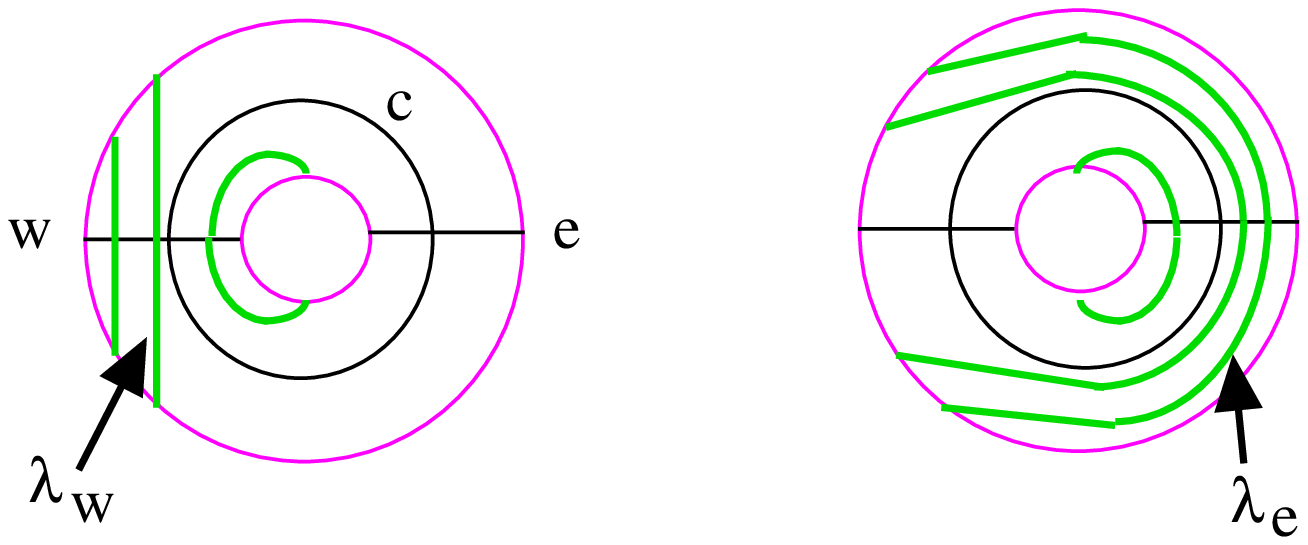}
    \caption{} \label{fig:Westoeast}
    \end{figure}

More generally, suppose that $c$ is a simple closed curve in a surface $P$ and $w$ is a properly embedded arc in $P$ that intersects $c$ once.  Suppose $\lambda$ is a properly embedded $1$-manifold in $P$ that is disjoint from $c$ and intersects $w$ transversally.  Then a small regular neighborhood $\eta(c \cup w) \subset P$ can be viewed as an annulus $A$ in which $w$ is a spanning arc, $\lambda$ intersects $A$ in proper arcs, each of which intersects $w$ once, and each of which is disjoint from $c$ and from a distant fiber of $\eta(c) \subset \eta(c \cup w) $.  Performing the operation above to $\lambda \cap A$ will be called {\em band-summing $\lambda$ to $c$ along $w$.}  See Figure \ref{fig:Bandsum}.

  \begin{figure}[tbh]
    \centering
    \includegraphics[scale=0.6]{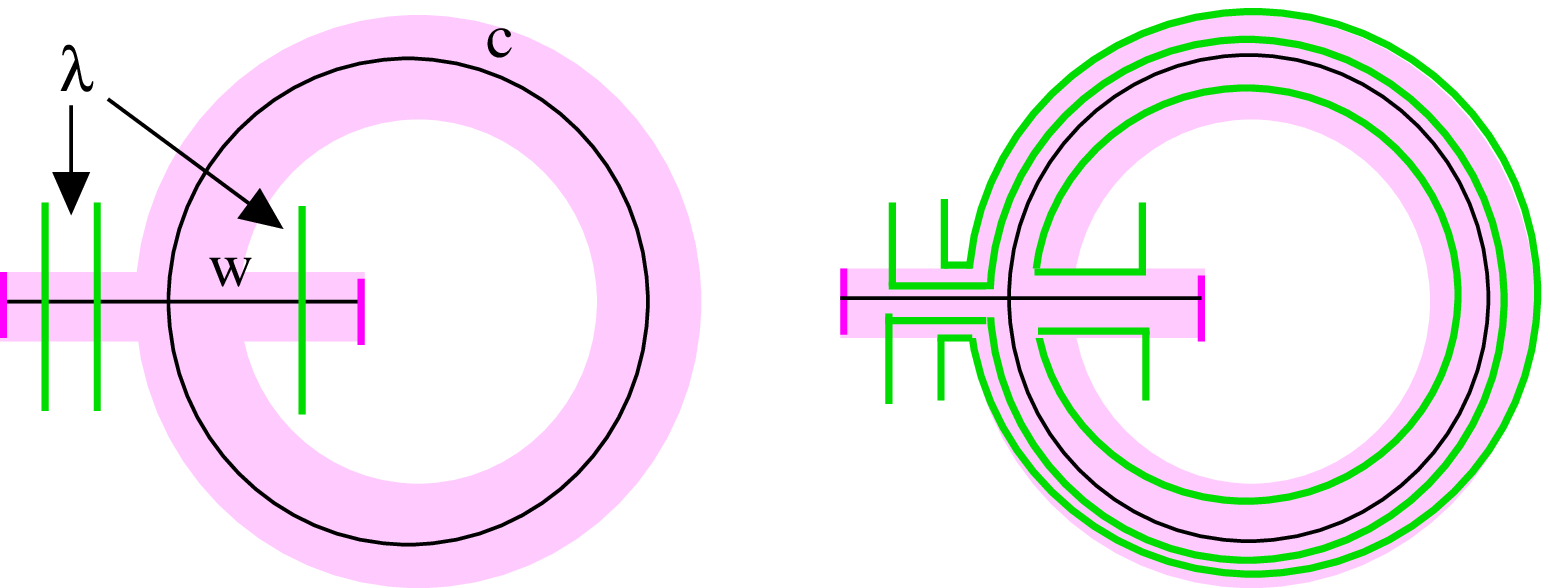}
    \caption{} \label{fig:Bandsum}
    \end{figure}

Here is an additional feature of this band-sum operation.  Suppose $M$ is a $3$-manifold and $P \subset  \bdd M$.  Suppose there are proper disks $C$ and $D$ in $M$ so that $\bdd C = c$ and $\bdd D = \lambda$.  Then after the operation, $\lambda$ still bounds a disk, one obtained by boundary-summing $\bdd D$ to one copy of $C$ for each point in $\lambda \cap w$.  This operation will be called {\em tube-summing $D$ to $C$ along $w$.}  See Figure \ref{fig:Tubesum}. 

  \begin{figure}[tbh]
    \centering
    \includegraphics[scale=0.6]{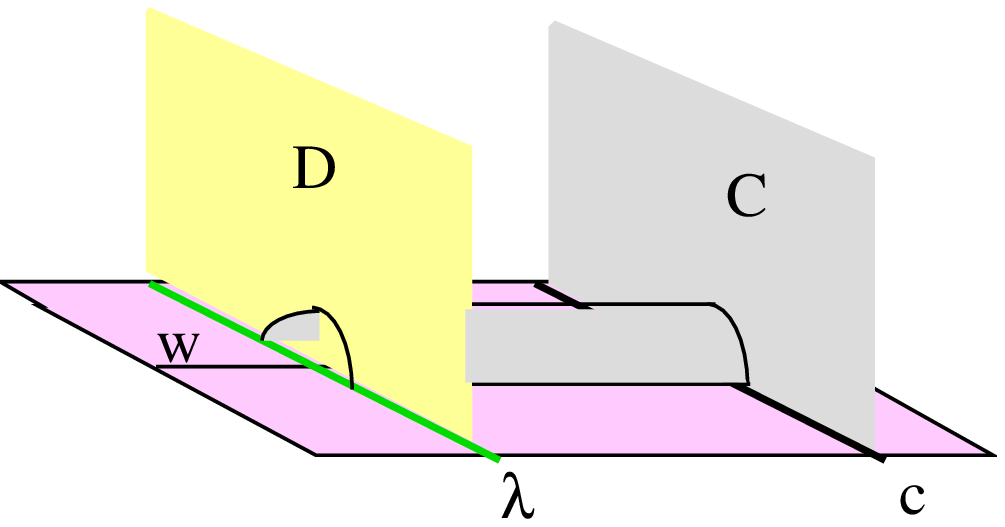}
    \caption{} \label{fig:Tubesum}
    \end{figure}

Now suppose $P$ is a compact orientable surface and $v, w \subset P$ are properly embedded arcs in $P$ that meet at a single point.  
Define a new orientable surface $P_{v - w}$ by the following operation: add a band to $P$ with its ends attached at the pair of points $\bdd v \subset \bdd P$.  Then remove a neighborhood of $w$.

$P$ and $P_{v - w}$ have the same Euler characteristic; whether they are homeomorphic or not then depends only on whether the operation changes the number of boundary components.  In any case, we have:

\begin{lemma} \label{lemma:Commute1}There is a homeomorphism $\phi_{v, w}: P_{w - v} \to P_{v - w} $ that is the identity away from $\eta(v \cup w)$.

\end{lemma}

  \begin{figure}[tbh]
    \centering
    \includegraphics[scale=0.6]{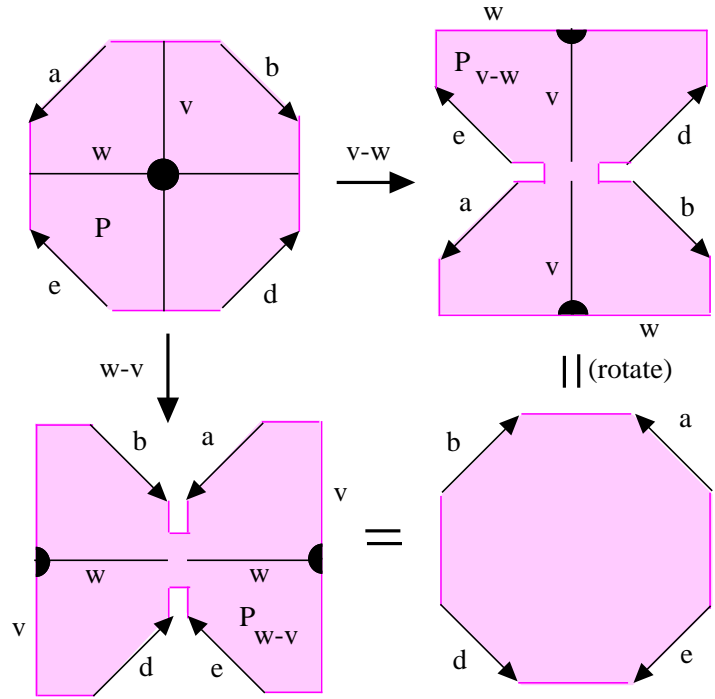}
    \caption{The arrows denote points of contact with the rest of the surface $P$.} \label{fig:Commute}
    \end{figure}

\begin{proof}  The proof is illustrated in Figure \ref{fig:Commute}.
\end{proof}

Suppose $\lambda$ is a properly embedded curve in $P$, in general position with respect to $w$ and disjoint from $v$.  Then $\lambda$ is unaffected by the operation that creates $P_{w - v}$.  This observation then provides a natural embedding $\lambda \subset P_{w - v}$. 

\begin{lemma} \label{lemma:Commute2} Let $P_+$ be the surface obtained from $P$ by adding a band to $P$ with its ends attached at the pair of points $\bdd v \subset \bdd P$.  Let $v_+$ be the circle in $P_+$ which is the union of $v$ and the core of the band.  Then $\phi_{v, w}(\lambda) \subset P_{v - w} \subset P_+$ is the curve obtained from $\lambda$ by band-summing $\lambda$ along $w$ to $v_+$.
\end{lemma}

  \begin{figure}[tbh]
    \centering
    \includegraphics[scale=0.6]{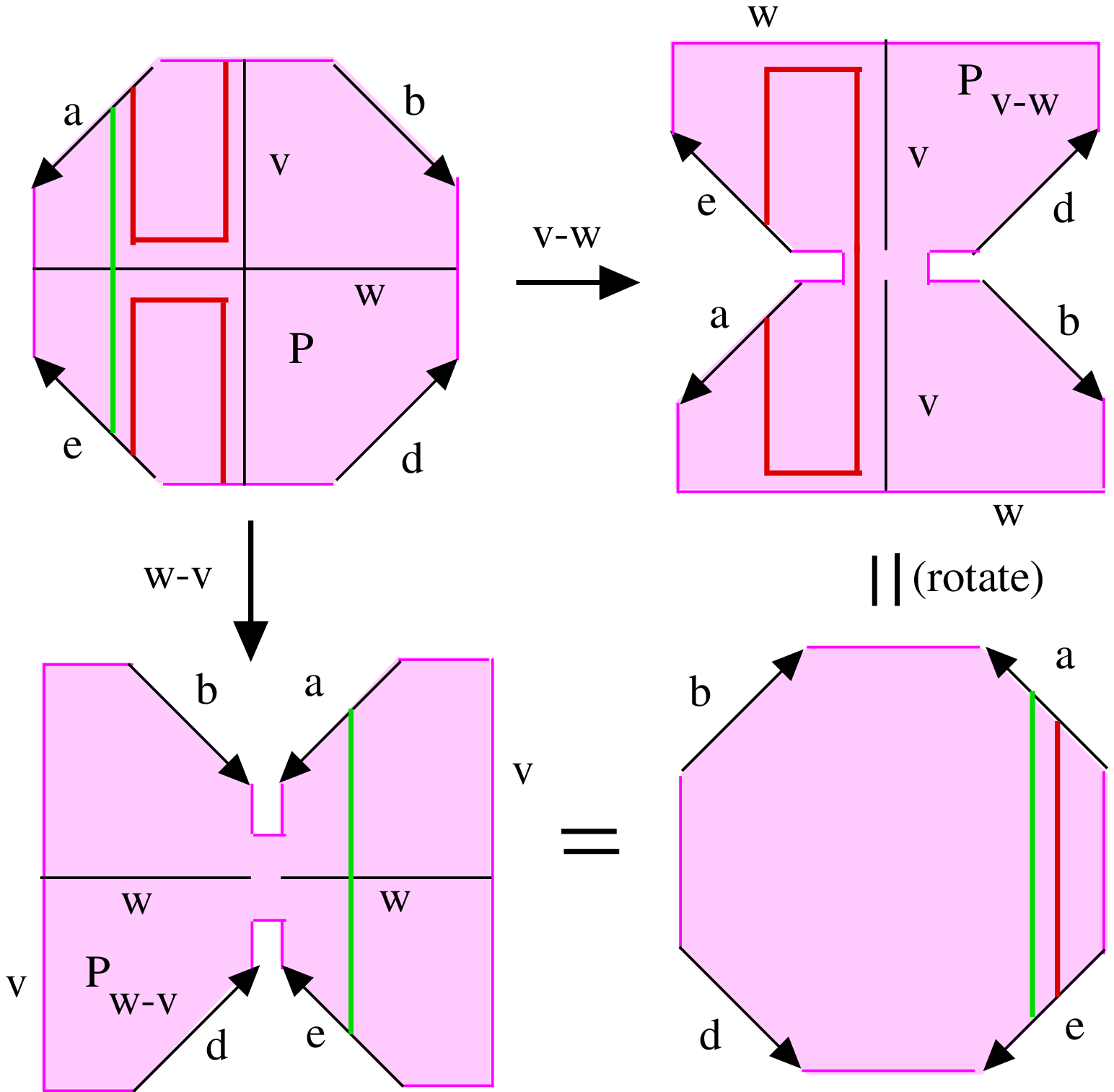}
    \caption{} \label{fig:Commutearcs}
    \end{figure}

\begin{proof}  The proof is illustrated in Figure \ref{fig:Commutearcs}.
\end{proof}

\section{The fundamental construction}

\begin{prop} \label{prop:fundamental} Suppose $M_+ = \cVp \cup_{S_+} \cWp$, $M_- = \cVm \cup_{S_-} \cWm$ are Heegaard splittings.  Suppose $P \subset S_+$ and $F \subset S_-$ are surfaces with respect to which a collection of disks $\oV \cup \oW$ and associated op-arcs $\cv \cup \cw$ is a stabilizing pair of coherently numbered coordinated forests of disks.  Suppose further that for some peripheral $(i, j)$ with $\rho_{\pm}(i, j) \neq 0$, $V_i \subset \cVp$ and $W_j \subset \cWm$ (or vice versa) and 
\begin{enumerate}
\item $\bdd V_i \cap P - v_i^+$ is disjoint from all op-arcs
\item $\bdd W_j \cap F - w_j^-$ is disjoint from all op-arcs
\item either $v_i^+$ or $w_j^+ \subset P$ is disjoint from all op-arcs and
\item either $v_i^-$ or $w_j^- \subset F$ is disjoint from all op-arcs.
\end{enumerate}
Then there are surfaces $P' \subset S_+$ and $F' \subset S_-$, with respect to which a collection of disks $\oV' \cup \oW'$ and associated op-arcs $\cv' \cup \cw'$ is a stabilizing coordinated pair of coherently numbered forests of disks.  Moreover, there are fewer disks in $\oV'$ than in $\oV$ and fewer disks in $\oW'$ than in $\oW$.
\end{prop} 

\begin{proof} We construct another stabilizing coordinated pair of coherently numbered forests of disks.   We describe the construction in $M_+$ and later note the effect of the symmetric construction in $M_-$.

Start with the surface $P'' \supset P$ that is the union of $P$ with a collar neighborhood $Y = \eta(\bdd V_i)$ of $\bdd V_i$ in $S_+$.  Since part of $\bdd V_i$ already lies in $P$, another way to view the construction of $P''$ from $P$ is to add to $P$ a band in $S_+ - P$ along each arc of $\bdd V_i - P \subset S_+$.  

We initially assume that $w_j^+$ is not a pair of op-arcs, but, like $v_i^+$,  just a single proper arc in $P$.  The arcs $v_i^+, w_j^+$ intersect in a single point, since by Corollary \ref{cor:periph}, $\rho_{\pm}(i, j) = 1.$  Since $(i, j)$ is peripheral, the arc $w_j^+$ may intersect other arcs $v_{\ell}^+$ but only if $\ell< i$.  
Band-sum all such $v_{\ell}^+$ along $w_j$ to  $\bdd V_i$ and call the result $v'^+_{\ell} \subset P''$.  If $v_{\ell}^+$ was on the boundary of a disk in $\oV$, tube-sum the disk to (copies of) $V_i$ to obtain a corresponding disk in $\oV'$.  If $v_{\ell}^+$ was a pair of op-arcs (so, by assumptions (1) and (3), $\bdd V_i$ is disjoint from all op-arcs $\cw$) then $v'^+_{\ell}$ is a pair of op-arcs in $P''$.  Although after this step $v'^+_{\ell}$ may not intersect all $w_j^+$ efficiently, it is straightforward to see that, when the pair $v'^+_{\ell}$ is isotoped in $P''$ to make all intersections efficient, the Separation, Parallel and Ordering Properties on the pair $v_{\ell}$ in $P$ induce the same properties on the pair $v'^+_{\ell}$ in $P''$.  New op-ties may have been introduced, each corresponding to an intersection point of some $w_k^+$ with an arc of $\bdd V_i \cap P$.  Now remove the original $V_i$ from the collection of disks and call the result $\oV'$.

After the operation described above, $w_j \subset P''$ is disjoint from all disks in $\oV'$ and from all op-arcs in $\cv'$.  Let $P' = P'' - \eta(w_j)$.  Augment the set of op-arcs $\cv'$ by adding the pairs of arcs $\bdd Y \cap P$, one pair $v'^+_k$ for each arc $v_k^+$ in $\bdd V_i \cap P - v_i^+$, and normally orient each $v'^+_k$ into $Y$.  The assumptions of the proposition guarantee that the new pair of op-arcs $v'^+_k$ is disjoint from all other op-arcs and it is easy to see from the construction that it satisfies the Separation and Parallel Properties.  See Figure \ref{fig:Peripheral}.

 \begin{figure}[tbh]
    \centering
    \includegraphics[scale=0.6]{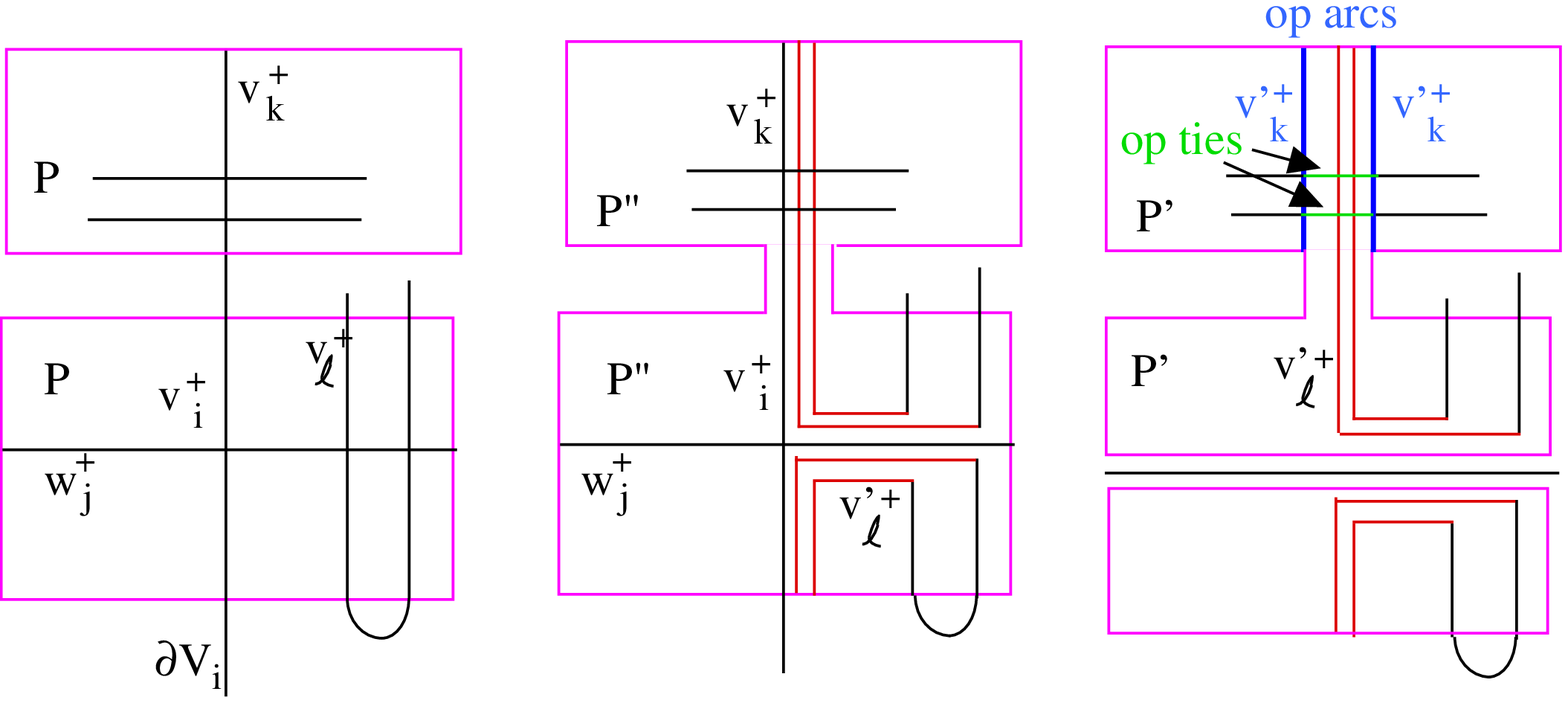}
    \caption{} \label{fig:Peripheral}
    \end{figure}

That such a pair of new op-arcs $v'^+_k$ satisfies the Order Property is only a little more difficult to show:  By the coherence of the numbering, $v_k^+ \subset \bdd V_i \cap P - v_i^+$ guarantees that $k > i$.  The interior of each op-tie of  the new pair of op-arcs $v'^+_k$ (corresponding to a point of $v_k^+ \cap \oW$) intersects only those $v'^+_{\ell}$ that have been band-summed to $v_i^+$ along $w_j^+$, that is only those for which $v_{\ell}^+ \cap w_j^+ \neq 0$.  Since $(i, j)$ is peripheral, this implies $\ell < i$, hence $\ell < k$, as required.  

If $w_j^+ = w_j^a \cup w_j^b$ is a pair of op-arcs, the construction is only slightly different.  By the Parallelism Property, points of intersection of $w_j^a$ with any $v_{\ell}^+$ are paired to points of intersection of $w_j^b$ by op-ties.  So the band summing described above, using say the component $w_j^a$, in fact removes (when the intersections are made efficient) all points of intersection between $\bdd \oV'$ and $w_j^b$ as well.  So then {\em both} op-arcs $w_j^a$ and $w_j^b$ end up disjoint from $\bdd \oV'$ and neighborhoods of both should be removed.  See Figure \ref{fig:Shield}.

 \begin{figure}[tbh]
    \centering
    \includegraphics[scale=0.6]{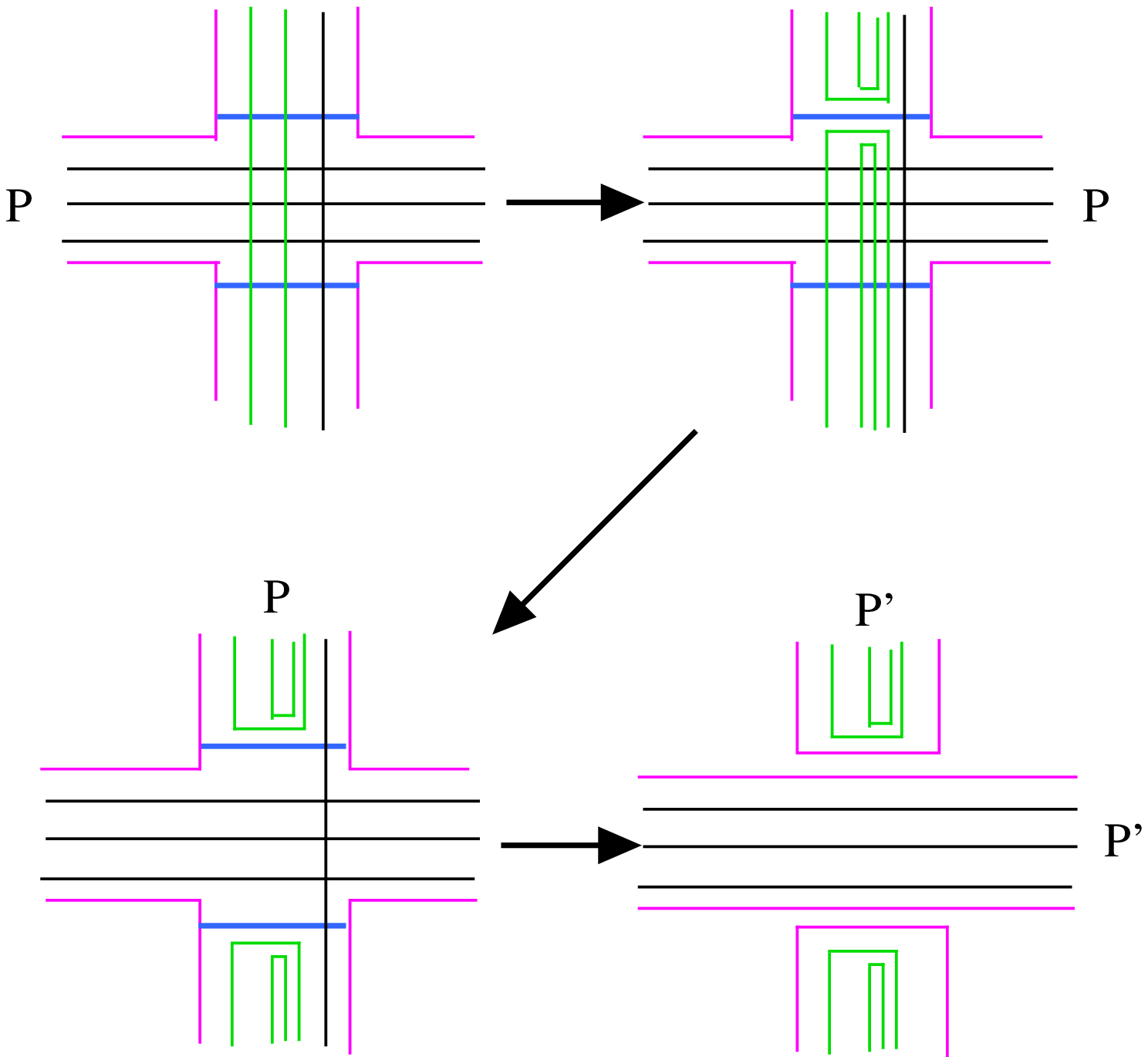}
    \caption{} \label{fig:Shield}
    \end{figure}

It is much more difficult to show that the new collection $\cv'$ of op-arcs still satisfies the Disk Property; that piece of the argument is postponed until later.  (See Corollary \ref{cor:Stilldisk}.)

\bigskip

What is the effect of the construction described above on the forest of trees?  Is the result a new pair of forests?  First of all, $w_j^+$ disappears, so, if $w_j^+$ is not a pair of op-arcs, and so lies on $\bdd W_h$ for some disk $W_h \subset \cWp$ with $h < j$, then $\bdd W_h$ has one less arc of intersection with $P'$.  Also, the disk $W_j \subset \cWm$ becomes the root of a tree with root arc $w_j^-$.  Secondly the entire disk $V_i^+$ disappears, so each disk (in $\cVm$) whose vertex, in the forest, was adjacent to $\aaa_i$ away from $v_i$, becomes a root in the resulting forest, a root associated to the new pairs of op-arcs that we have created.    But there are two immediately apparent defects: The arc or pair of op-arcs $v_i^- \subset F$ no longer has a matching arc $v_i^+ \subset P$, since $V_i$ has been removed.  
Also, $w_j^+$ has been removed, whereas $w_j^- \subset \bdd W_j$ remains as a root arc, violating the condition that each root arc in $F$ is coordinated with a pair of op-arcs in $P$.

Both defects are overcome by doing the symmetric operation in $M_-$ using now $\rho_-(i, j) = 1$.  That is, tube-sum disks in $\cWm$ along $v_i^-$ to $W_j$, alter $F$ by removing a neighborhood of $v_i^- \subset F$ (thereby fixing the first defect) and add to $F$ a neighborhood of the arcs $\bdd W_j^- - F$.  Then delete the disk $W_j$, fixing the second defect. See Figure \ref{fig:Peripheral2}.

 \begin{figure}[tbh]
    \centering
    \includegraphics[scale=0.6]{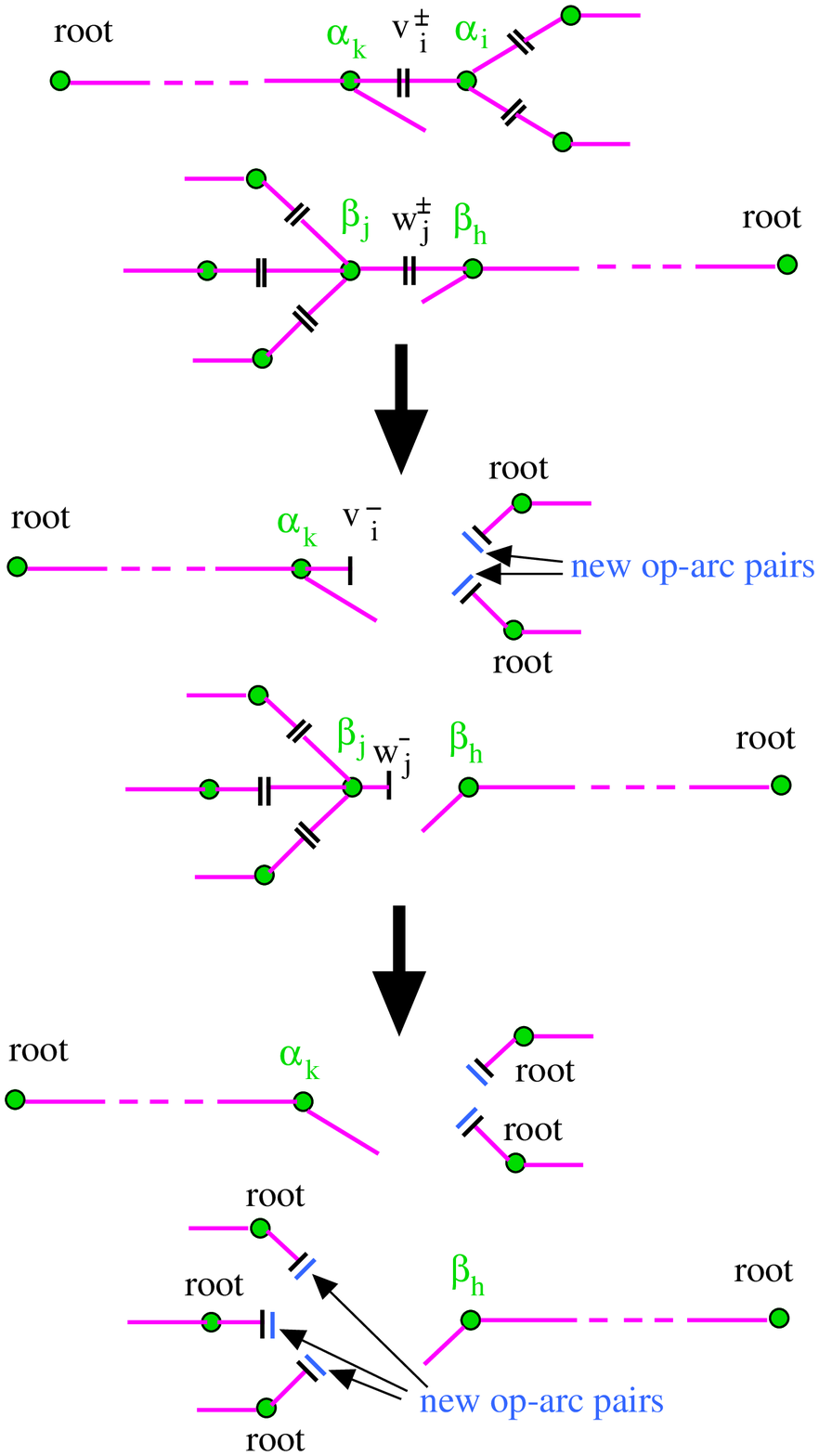}
    \caption{} \label{fig:Peripheral2}
    \end{figure}
    
We have shown that the new surfaces $P'$ and $F'$ and the new forests of disks satisfy all of the properties (except perhaps the Disk Property) of a coherently numbered stabilizing pair of forests of disks.  The new forests have fewer disks since the disks $V_i, W_j$ (corresponding to vertices $\aaa_i$ and $\bbb_j$ in the two forests) have been removed. 

We now assume that the new framework also satisfies the Disk Property (we will show this later) and verify that then the forests are coordinated.  That is, 

\begin{lemma} \label{lemma:ground}  For $\sss'_{\pm}$ the new pairings in $\hP'$ and $\hF'$ constructed as above and
for each $(\ell, k) \in \No \times \No$, $\sss'_+(\ell, k) = \sss'_-(\ell, k)$.
\end{lemma}

\begin{proof}  Since the initial forests are coordinated, the statement is true before the construction. So the proof consists of showing that the construction process does not alter the relationship.

Whether or not any of the arcs $v_i^{\pm}$ or $w_j^{\pm}$ are op-arcs, all disappear from our accounting by the end of the construction, so they are irrelevant to the question.  The focus is on other arcs, which may change during the construction.  The curves that are altered (as $P$ also is altered) in $S_+$ are the curves $v_{\ell}^+$ which intersect $w_j^+$; those altered (as $F$ also is altered) in $S_-$ are the curves $w_{k}^-$ which intersect $v_i^-$.  

\bigskip

By Lemma \ref{lemma:sigmaone}  each number is either $0$ or $1$, so it suffices to prove 

{\bf Claim:}  For all $\ell \neq i, k \neq j$ in $\No$, $$\sss'_{\pm}(\ell, k) \cong \sss_{\pm}(\ell, k) + \sss_{\pm}(\ell, j) \cdot \sss_{\pm}(i, k) \; mod \; 2.$$  

{\bf Proof of Claim:}  By symmetry, it suffices to show that this is true in $S_+$, that is for the intersection pairing $\sss'_{+}$ on arcs in $\hP'$.

Following Lemma \ref{lemma:efficient} there is a way to accurately calculate $\sss_+(\ell, k)$ in $P$.  An intersection point of $v_{\ell}^+$ with $w_k^+$ counts as a point in $\sss_+(\ell, k)$ if and only if the point is not on any overpass from either $\cv$ or $\cw$, that is the intersection point lies on a ground arc of both  $v_{\ell}^+$ and $w_k^+$.  Similarly,  {\em once $v'^+_{\ell}$ and $w'^+_k$ are isotoped rel boundary to intersect efficiently}, an intersection point of $v'^+_{\ell}$ with $w'^+_k$ counts as a point in $\sss'_+(\ell, k)$ if and only if the point is on ground arcs of both $v'^+_{\ell}$ and $w'^+_k$.  

Since, to prove the claim, we only have to determine the parity of $\sss'_{\pm}(\ell, k)$, the requirement that the arcs $v'^+_{\ell}$ and $w'^+_k$ first be isotoped to intersect efficiently turns out to be irrelevant, as we now demonstrate.  Two proper arcs in a surface can be isotoped to intersect efficiently by a sequence of isotopies, each removing a bigon of intersection.  So, to demonstrate that this process does not change the parity of intersection points between ground arcs in $v'^+_{\ell}$ and ground arcs of $w'^+_k$, it suffices to show that for any bigon $B$ in $P'$ between a subarc $\aaa$ of $v'^+_{\ell}$ and a subarc $\bbb$ of $w'^+_k$, either both end points of $\aaa$ lie in ground arcs of $v'^+_{\ell}$ or neither does (and symmetrically for $w'^+_k$).  It follows from the Separation Property that any subarc of $v'^+_{\ell}$ that has one end off an overpass and one end on must intersect the associated pair of op-arcs an odd number of times.  On the other hand, since the op-arcs $\cw$ are disjoint from $\bbb \subset w'^+_k$, any subarc of $\cw$ that lies in $B$ must have both ends on $\aaa$.  That is, $\aaa$ intersects each op-arc in $\cw$ an even number of times.  Hence if one end of $\aaa$ is on any overpass so is the other.

Since we do not have to make $v'^+_{\ell}$ and $w'^+_k$ intersect efficiently, we need only count (parity of) points of intersection as they are originally constructed.  We have already seen that the Ordering Property and the fact that $(i, j)$ is peripheral guarantees that the point $x = v_i^+ \cap w_j^+$ is a ground arc in both $v_i^+$ and $w_j^+$.  

\bigskip

{\bf Case 1:}  $v_i$ is disjoint from all pairs of op-arcs in $\cw$.

In this case, since $x$ is on a ground arc of $v_i$, all of $v_i$ is a ground arc. Suppose an intersection point $y \in v_{\ell}^+ \cap w_j^+$ is on an overpass associated to a pair of op-arcs $v_s^+$.  Then each of the pair of op-arcs $v_s^+$ also intersects $w_j^+$, namely at the ends of the op-tie on which $y$ lies.  When these three arcs ($v_{\ell}^+ \cup v_s^+$) are band-summed to $v_i^+$ the resulting subarc of $v'^+_{\ell}$ still lies entirely on the overpass associated to the new pair $v'^+_s$.  So whatever intersections of $v'^+_{\ell}$ with $w_k$ are created by this tube-summing do not count in $\sss'(\ell, k)$.  Hence in calculating how $\sss'_+(\ell, k)$ differs from $\sss_+(\ell, k)$ we can focus only on points of $v_{\ell}^+ \cap w_j^+$ that lie in ground arcs of $w_j^+$.  Similarly, we can focus only on points that lie in ground arcs of $v_{\ell}^+$ since if $y$ is on an op-tie for some pair of op-arcs $w_t^+$, band-summing $v_{\ell}^+$ near $y$ to $v_i^+$ only creates a much longer tie, since $v_i$ is disjoint from the op-arcs $w_t^+$; new points of intersection don't count in $\sss'_+(\ell, k)$.

So the only relevant change caused by the construction could come from band-summing $v_{\ell}^+$ near a point $z$ (unique, if it exists, by Lemma \ref{lemma:sigmaone}) in $v_{\ell}^+ \cap w_j^+$ that lies on the ground arc of both curves.  If no such point exists, then $\sss_+(\ell, j) = 0$ and the number of intersection points of $v'^+_{\ell} \cap w_k^+$ that lie in ground arcs of each is unchanged.  That is, $$\sss'_+(\ell, k) = \sss_+(\ell, k) = \sss_+(\ell, k) + 0 \cdot \sss_+(i, k) = \sss_+(\ell, k) + \sss_+(\ell, j) \cdot \sss_+(i, k)$$ as required.  If $z \in v_{\ell}^+ \cap w_j^+$ does lie on a ground arc of each, then the construction band-sums the ground arc of $v_{\ell}^+$ to $v_i^+$ at $x$.  It follows that the number of intersection points of ground arcs of $v'^+_{\ell}$ with ground arcs of $w_k^+$ is increased by $\sss_+(i, k)$ (before $v'^+_{\ell}$ is isotoped to have efficient intersection with $w_k^+$).  That is,
 $$\sss'_+(\ell, k) = \sss_+(\ell, k) + 1 \cdot \sss_+(i, k) = \sss_+(\ell, k) + \sss_+(\ell, j) \cdot \sss_+(i, k) \; mod \; 2$$ as required.
 
 \bigskip
 
 {\bf Case 2:} $w_j$ is disjoint from all pairs of op-arcs in $\cv$

The proof is quite analogous to Case 1.  Here, since $x$ is on a ground arc of $w_j$, all of $w_j$ is a ground arc.  If $y \in v_{\ell}^+ \cap w_j^+$ is not on a ground arc of $v_{\ell}^+$, consider the op-tie in $v_{\ell}^+$ on which $y$ lies, say for a pair of op-arcs $w_t^+$.  Observe first the subtle fact that $w_t^+$ must be disjoint from $v_i^+$.  For if it weren't, there would be an op-tie for $w_t^+$ contained in $v_i^+$, and by the Parallelism Property that op-tie also must intersect $w_j^+$ and so $\rho_+(i, j) \geq 2$, contradicting Corollary \ref{cor:periph}.  It follows then that when the op-tie in $v_{\ell}^+$ containing $y$ is band-summed to $v_i^+$, the resulting arc becomes an op-tie in $v'^+_{\ell}$ for the pair of op arcs $w_t^+$.  Thus none of the new points introduced affects $\sss'(\ell, k)$.  So, as in Case 1, we need only focus on the point $z$ (unique, if it exists) at which a ground arc of $v_{\ell}^+$ intersects $w_j^+$.  The rest of the argument is essentially the same as in Case 1.  This proves the Claim, and so (assuming the Disk Property is preserved by the construction) Lemma  \ref{lemma:ground} and with it Proposition \ref{prop:fundamental}.  
\end{proof}  \end{proof}

\section{The Disk Property is preserved}

We want to understand how the fundamental construction, described in the proof of Proposition \ref{prop:fundamental} above, that changes $P$ to $P'$ affects the topology of the surfaces $\hP$ and $\hP'$ obtained by building all overpasses in $P$ and $P'$.  The operation $P \to P_{v - w}$ described in Section \ref{sec:surfaces} plays a role:

\begin{lemma}  $\hP' = \hP_{\hv^+_i - \hw^+_j}$.
\end{lemma}

\begin{proof} The first observation is this:  The (abstract) surface obtained from $P'$ by building exactly those overpasses that are newly created in $P'$ is simply $P_{v_i^+ - w_j^+}$.  This is immediate from the description: when the overpass is built for the overpass corresponding to an arc $v_{i'}^+$  of $\bdd V_i \cap P - v_i^+$, the effect on the topology of $P'$ is as if $v_{i'}^+ \subset \bdd V_i$ were simply disjoint from $P$.  Apply that logic to every component of $\bdd V_i \cap P - v_i^+$, and so to every newly created overpass, and the effect is as if the entire arc $\bdd V_i - v_i^+$ were disjoint from $P$.  That is, once all  the new overpasses are built, it is as if a single band were attached to $P$ with core the arc $\bdd V_i - v_i^+$, and then the arc $w$ is deleted.  This is the same description as the surface $P_{v_i^+ - w_j^+}$.

{\bf Case 1:} $w_j^+$ is disjoint from all op-arcs $\cv$ and is a simple arc (not a pair of op-arcs).

$\hP'$ is obtained from $P'$ by building all overpasses.  Build the new overpasses first, changing $P'$ to $P_{v_i^+ - w_j^+}$.  Since all the remaining op-arcs are unaffected by removing $w_j^+$ they persist into $P_{v_i^+ - w_j^+}$ and $\hP'$ can be viewed as the result of building the overpasses in $P_{v_i^+ - w_j^+}$.  By the hypothesis of this case, none of the old overpasses goes through the band attached at the ends of $v_i^+$ so we may as well attach it, and remove $w_j^+$ {\em after} building the old overpasses.  But this is equivalent to first creating $\hP$ (by building the old overpasses) then attaching the band to the ends of what was $v_i^+$ and is now $\hv^+_i$ and then removing $w_j^+ = \hw^+_j$.

{\bf Case 2:} $w_j^+$ is a pair of op-arcs and so is disjoint from $\cv$.

The argument is much the same as Case 1, but requires a preliminary move:  before launching the argument above, first build the overpass corresponding to $w_j^+$, creating a surface $P_j$ that plays the role of $P$ in Case 1.  $w_j^+$ becomes a single arc $w_j$ in $P_j$ intersecting $v_i^+$ in a single point and removing $w_j$ from $P_j$ gives the same surface as removing both of the original op-arcs $w_j^+$ from $P$.

{\bf Case 3:} $v_i^+$ is disjoint from all op-arcs $\cw$.

The important difference from Case 1 is that here the op-arcs $\cv$ may intersect $w_j^+$ in $P$; during the construction of $P'$ they are rerouted.  Begin the construction the same as in Case 1: build all new overpasses, so that $P'$ becomes $P_{v_i^+ - w_j^+}$.  The old op-arcs that previously intersected $w_j^+$ are rerouted through the new band via the same operation that is described in Lemma \ref{lemma:Commute2}.  So, according to that Lemma, an equivalent way of viewing the surface at this point would have been to construct  $P_{ w_j^+ - v_i^+}$, leaving the op-arcs where they are, disjoint from $v_i^+$ and then apply $\phi_{v_i^+, w_j^-}$.  Then the argument of Case 1 applied to $P_{ w_j^+ - v_i^+}$ shows that $\hP' = \hP_{ \hw_j^+ - \hv_i^+}$ and Lemma \ref{lemma:Commute1} shows that $\hP_{ \hw_j^+ - \hv_i^+} \cong \hP_{\hv^+_i - \hw^+_j}$. 
\end{proof}

\begin{cor} \label{cor:Stilldisk} $\hP'$ is a disk.
\end{cor}

\begin{proof}  We are given before the construction that $\hP$ is a disk, and for $v, w$ any two properly embedded arcs in a disk $D$ that intersect in a point, $D_{v - w}$ is a disk. 
\end{proof}

\section{Dropping symmetry: a combinatorial proof of the Gordon Conjecture}

\begin{prop} \label{prop:asymmetry} Suppose $M_+ = \cVp \cup_{S_+} \cWp$ and $M_- = \cVm \cup_{S_-} \cWm$ are Heegaard splittings.  Suppose collections of disks $\oV, \oW$ and associated op-arcs $\cv,  \cw$ is a stabilizing pair of coherently numbered coordinated forests of disks for surfaces $P \subset S_+$ and $F \subset S_-$.  Suppose further that there is a peripheral $(i, j)$ with  $\rho_{\pm}(i, j) \neq 0$ and that all op-arcs $\cv$ are disjoint from all arcs $\{ w_k^{\pm} \}$ in both $F$ and $P$.  (Note: but {\em not} symmetrically: That is, op-arcs in $\cw$ may intersect $\{ v_i^{\pm} \}$.)

If neither $M_+ = \cVp \cup_{S_+} \cWp$ nor $M_- = \cVm \cup_{S_-} \cWm$ is stabilized then there are surfaces $P' \subset S_+$, $F' \subset S_-$, and collections of disks $\oV', \oW'$ and associated op-arcs $\cv', \cw'$ so that
\begin{itemize}
\item $\oV', \oW'$ and associated op-arcs $\cv', \cw'$ is a stabilizing pair of coherently numbered coordinated forests of disks with respect to $P'$ and $F'$ and
\item there are fewer disks in $\oV'$ than in $\oV$ and fewer disks in $\oW'$ than in $\oW$ and 
\item all op-arcs $\cv'$ are disjoint from all arcs $\{ w'^+_{\pm} \}$ in both $F'$ and $P'$.
\end{itemize}
\end{prop} 

\begin{proof}  Among all $(i, j)$ with $\rho_{\pm}(i, j) \neq 0$ choose that in which $i$ is maximal.  If $V_i$ and $W_j$ both lie in the same manifold, say $M_+$ then Lemma \ref{lemma:terminate} and Corollary \ref{cor:periph} show that the splitting of $M_+$ is stabilized.  So henceforth we assume, with no loss of generality, that  $V_i \subset \cVp$ and $W_j \subset \cWm$.

Since $i$ was chosen to be maximal among non-trivial peripheral pairs $(i, j)$, each arc $v^+_{i'} \subset \bdd V_i \cap P - v_i^+$ is disjoint from all arcs $\{ w_k^+ \}$, else a maximal $k$ with non-trivial intersection would be a peripheral pair with $i' > i$.  It follows that the first requirement of Proposition \ref{prop:fundamental}, namely that $\bdd V_i \cap P - v_i^+$ is disjoint from all op-arcs, is satisfied.  All three other requirements are trivially satisfied, since any $\cv$ is disjoint from all all arcs $\{ w'^+_k \}$.  Hence we can apply the fundamental construction to the pair of disks $V_i$ and $W_j$ as was done in the proof of Proposition \ref{prop:fundamental}.  New op-arcs are created in $\cv'$, one pair for each arc $v^+_{i'}$ of $\bdd V_i \cap P - v_i^+$.  But we have observed above that our choice of $i$ guarantees that each of these will be disjoint from all arcs $\{ w_k^+ \}$, as required.
\end{proof}

\begin{thm} If $\cV \cup_S \cW$ is stabilized either $\cVp \cup_{S_+} \cWp$ or $\cVm \cup_{S_-} \cWm$ is stabilized.
\end{thm}

\begin{proof}  Begin with the Seminal Example of a stabilizing pair of coherently numbered coordinated forests of disks. Since there are no op-arcs in this example, it clearly satisfies the hypotheses of Proposition \ref{prop:asymmetry}.  Repeatedly apply Proposition \ref{prop:asymmetry}, stopping if some iteration shows that one of $M_+ = \cVp \cup_{S_+} \cWp$ or $M_- = \cVm \cup_{S_-} \cWm$ is stabilized.  Since each application decreases the number of indices represented by disks in $\oV$ and $ \oW$, Proposition \ref{prop:asymmetry} can only be applied a finite number of times.  If the process does not stop because it detects a stabilized splitting, it must stop because there are no longer any peripheral $(i, j)$ for which $\rho_{\pm}(i, j) \neq 0$.  This implies that $\rho_{\pm} = 0$.  

In this case, consider the disks $V_0, W_0$ that define the distinguished roots.  They are both contained in $M_+$ or both in $M_-$, since their boundaries have the common intersection point $x_0$ outside of $P$ or $F$.  Say both are contained in $M_+$.  The arcs $\bdd V_0 \cap P$ are disjoint from the arcs $\bdd W_0 \cap P$ since $\rho_+ = 0$.  Hence the only intersection point in $\bdd V_0 \cap \bdd W_0$ is $x_0$.  Thus $V_0$ and $W_0$ are a stabilizing pair of disks for the splitting $\cVp \cup_{S_+} \cWp$. 
\end{proof}

 \bibliography{mybibliounique}
 \bibliographystyle{plain}

\end{document}